\documentclass[11pt,a4paper]{article}
\usepackage{amssymb,amsmath}

\newtheorem{theorem}{Theorem}[section]
\newtheorem{corollary}[theorem]{Corollary}
\newtheorem{definition}[theorem]{Definition}
\newtheorem{lemma}[theorem]{Lemma}
\newtheorem{proposition}[theorem]{Proposition}

\newenvironment{proof}{\begin{trivlist}\item[]{\it
Proof.}}{\hfill$\square$\end{trivlist}}

\def\mc{{\mathbb{C}}}
\def\mn{{\mathbb{N}}}
\def\gl{{\mathrm{GL}}}
\def\ooo{{\mathrm{O}}}
\def\monoms{{\mathcal{M}}}
\def\freeinvar{{\mathcal{F}}}
\def\relthree{J}
\def\relone{J}
\def\reltwo{J}
\def\partitions{{\mathrm{Par}}}
\def\unipotent{{\mathrm{U}}}
\def\torus{{\mathbb{T}}}
\def\diag{{\mathrm{diag}}}
\def\mz{{\mathbb{Z}}}
\def\Span{{\mathrm{Span}}}
\def\mult{{\mathrm{mult}}}
\def\basis{{\mathcal{G}}}
\def\primary{P}
\def\secondary{{\mathbf{S}}}
\def\trysecondary{S}
\def\mingen{Q}

\begin{document}
\title{Multisymmetric polynomials in dimension three} 
\author{M\'aty\'as Domokos ${}^a$ \thanks{
Partially supported by 
OTKA NK72523 and K61116.}\; 
\quad and Anna Pusk\'as ${}^b$
\\ 
\\ 
{\small ${}^a$ R\'enyi Institute of Mathematics, Hungarian Academy of 
Sciences,} 
\\ {\small 1053 Budapest, Re\'altanoda utca 13-15., Hungary } 
\\{\small E-mail: domokos@renyi.hu } 
\\ 
\\
{\small ${}^b$ Columbia University, Department of Mathematics,}
\\ {\small MC 4406, 2990 Broadway,  New York, NY 10027, USA.}  
\\ {\small E-mail:  apuskas@math.columbia.edu}
}
  
\date{}
\maketitle 
\begin{abstract} 
The polarizations of one relation of degree five and two relations of degree six minimally generate the ideal of relations among a minimal generating system of the algebra of multisymmetric polynomials 
in an arbitrary number of three-dimensional vector variables. 
In the general case of $n$-dimensional vector variables, a relation of degree $2n$ among the polarized power sums is presented such that it is not contained in the ideal generated by lower degree relations. 
\end{abstract}



\section{Introduction}\label{sec:intro} 

The symmetric group $S_n$ acts on the $n$-dimensional complex vector space 
$V:=\mc^n$ by permuting coordinates. Consider the diagonal action of $S_n$ on the space 
$V^m:=V\oplus\cdots\oplus V$ of $m$-tuples of vectors from $V$. 
The {\it algebra of multisymmetric polynomials} is the corresponding ring of invariants 
$R_{n,m}:=\mc[V^m]^{S_n}$, consisting of the polynomial functions on $V^m$ that are constant along the $S_n$-orbits.  

In the special case $m=1$, $R_{n,1}$ is a polynomial ring generated by the elementary symmetric polynomials (or by the first $n$ power sums). It is classically known (see \cite{schlafli}, 
\cite{macmahon}, \cite{weyl}) that the polarizations of the elementary symmetric polynomials constitute a minimal $\mc$-algebra generating system of $R_{n,m}$ for an arbitrary $m$. 
The ideal of relations among these generators is not completely understood, although it was classically studied in \cite{junker1}, \cite{junker2}, \cite{junker3}, and in a couple of more recent papers (see the references in Section~\ref{sec:prel}). 

An explicit finite presentation of $R_{n,m}$ by generators and relations is known,  
see Proposition~\ref{prop:general} below. 
Note that the price for having a uniform description of the ideal of relations in Proposition~\ref{prop:general} is the inclusion of redundant elements in the system of generators. 

In the present paper for $n=3$ we determine a minimal system of generators of the ideal of relations among a minimal generating system of $R_{3,m}$. We exploit the natural action of the general linear group $\gl_m$ on $R_{3,m}$. 
Identify $V^m$ with the space $\mc^{n\times m}$ of $n\times m$ matrices. 
The complex general linear group $\gl_m$ acts on $\mc^{n\times m}$ by matrix multiplication from the right:  
$x\mapsto xg^{-1}$ ($x\in\mc^{n\times m}$, $g\in\gl_m$). As usual, this induces 
an action of $\gl_m$ on the coordinate ring $\mc[V^m]$ by linear substitution of variables. 
Since the action of $\gl_m$ on $V^m$ commutes with the action of $S_n$, the algebra $R_{n,m}$ is a 
$\gl_m$-submodule in the coordinate ring $\mc[V^m]$. 
In particular, we may choose a minimal system of homogeneous $\mc$-algebra generators of $R_{n,m}$ whose $\mc$-linear span is a $\gl_m$-submodule $W_{n,m}$ in $\mc[V^m]$. Write $S(W_{n,m})$ for the symmetric tensor algebra of $W_{n,m}$. 
(This is polynomial ring, with one variable associated to each element of a fixed basis of $W_{n,m}$.)   
Endow the algebra $S(W_{n,m})$ with the $\gl_m$-module structure induced by the representation of $\gl_m$ on $W_{n,m}$. 
Consider the natural $\mc$-algebra surjection $\varphi:S(W_{n,m})\to R_{n,m}$ extending the identity map on $W_{n,m}$. 
The ideal $\ker(\varphi)$ of relations among the chosen generators of  $R_{n,m}$ is a $\gl_m$-submodule of $S(W_{n,m})$. 

The coordinate ring $\mc[V^m]=\mc[x_{ij}\mid i=1,\ldots,n; j=1,\ldots,m]$ is an $nm$-variable polynomial algebra, where $x_{ij}$ stands for the $i$th coordinate function on the $j$th vector component.  
Given a monomial $w=x_1^{\alpha_1}\cdots x_m^{\alpha_m}$ in the $m$-variable polynomial algebra 
$\mc[x_1,\ldots,x_m]$, 
set 
$$[w]:=\sum_{i=1}^nx_{i1}^{\alpha_1}\cdots x_{im}^{\alpha_m}.$$ 
These elements of $R_{n,m}$ are called the {\it polarized power sums}, and 
\begin{equation}\label{eq:min-gen-sys}
\{[x_1^{\alpha_1}\cdots x_m^{\alpha_m}]\mid \sum_{j=1}^m\alpha_j\leq n\}\end{equation}  
is a minimal system of $\mc$-algebra generators of $R_{n,m}$.  
 
Denote by $W_{n,m}$ the subspace of $R_{n,m}$ spanned by the set (\ref{eq:min-gen-sys}). 
This is a $\gl_m$-submodule.  
In the case $n=3$, we present three explicit elements in the kernel of the surjection 
$\varphi:S(W_{3,m})\to R_{3,m}$ such that each of them generates an irreducible $\gl_m$-submodule in 
$S(W_{3,m})$, and the union of any $\mc$-bases of these three irreducible $\gl_m$-submodules constitutes a minimal generating system of the ideal $\ker(\varphi)$ (see Theorem~\ref{thm:main}). 
 
 For arbitrary $n$ we point out a connection between multisymmetric polynomials and vector invariants of the full orthogonal group, and use this to show that a homogeneous system of generators of the ideal of relations between the polarized power sums must contain a relation of degree $2n$ (see Theorem~\ref{thm:lowerbound} for the precise statement). 
  

\section{Preliminaries}\label{sec:prel} 

Denote by $\monoms_m$ the set of monomials in the polynomial algebra 
$\mc[x_1,\ldots,x_m]$, and for a natural number $d$ denote by $\monoms_m^d$ the subset of monomials of degree at most $d$. To each $w\in\monoms_m$ associate an indeterminate $t(w)$, and take the commutative polynomial algebra $\freeinvar_{n,m}:=\mc[t(w)\mid w\in\monoms_m]$ in infinitely many variables. 
For each $d\in\mn$ it contains the subalgebra $\freeinvar_{n,m}^d:=\mc[t(w)\mid w\in\monoms_m^d]$. 
In particular, we identify $\freeinvar_{n,m}^n$ and $S(W_{n,m})$ in the obvious way: by definition, 
$\{[w]\mid w\in\monoms_m^n\}$ is a $\mc$-vector space basis of $W_{n,m}$, and 
the map 
$[w]\mapsto t(w)$ extends uniquely to a $\mc$-algebra isomorphism $S(W_{n,m})\cong \freeinvar_{n,m}^n$. 
We denote by $\varphi_{n,m}$ the surjection 
$\varphi_{n,m}:\freeinvar_{n,m}\to R_{n,m}$ given by $\varphi(t(w))=[w]$ for all $w\in\monoms_m$. 
The restriction of $\varphi_{n,m}$ to $\freeinvar_{n,m}^d$ will  be denoted by $\varphi_{n,m}^d$; it is a surjection onto $R_{n,m}$ whenever $d\geq n$. 
To simplify notation later in the text, we sometimes write $\freeinvar$ instead of $\freeinvar_{n,m}^n$ 
and $\varphi$ instead of $\varphi_{n,m}^n$. 

We recall some elements in the kernel of $\varphi_{n,m}$. 
By a {\it distribution} of the set $\{1,\ldots,n+1\}$ we mean a set $\pi:=\{\pi_1,\ldots,\pi_h\}$ of pairwise disjoint non-empty subsets whose union is $\bigcup_{i=1}^h\pi_i=\{1,\ldots,n+1\}$. 
Write $D_{n+1}$ for the set of distributions of $\{1,\ldots,n+1\}$.  
Take monomials $w_1,\ldots,w_{n+1}\in\monoms_m$, and set 
\begin{equation}\label{eq:fundrel}\Psi_{n+1}(w_1,\ldots,w_{n+1})=
\sum_{\pi\in D_{n+1}}
\prod_{\pi_i\in\pi}(-1)(|\pi_i|-1)!\cdot t(\prod_{s\in\pi_i} w_s).\end{equation}

\begin{proposition}\label{prop:general} 
The kernel of the surjection 
$\varphi_{n,m}^{n^2-n+2}:\freeinvar_{n,m}^{n^2-n+2}\to R_{n,m}$ 
is generated as an ideal by the elements 
$\Psi_{n+1}(w_1,\ldots,w_{n+1})$, 
ranging over all choices of monomials $w_i\in\monoms_m$ with $\deg(w_1\cdots w_{n+1})\leq n^2-n+2$.   
\end{proposition} 

Proposition~\ref{prop:general} is a special case of a result in \cite{domokos} dealing with vector invariants of a class of complex reflection groups. In loc. cit. we first gave a simple short proof of an infinite version about the kernel of $\freeinvar_{n,m}\to R_{n,m}$ (see also   \cite{berele}, \cite{bukhshtaber-rees}, \cite{dalbec}, \cite{vaccarino} for related work). Then we applied 
Derksen's general degree bound for syzygies from 
\cite{derksen} and ideas of Garsia and Wallach from \cite{haiman} to derive in particular the above finite presentation of $R_{n,m}$ in \cite{domokos}. 

To produce elements in $\ker(\varphi)=\ker(\varphi_{n,m}^n)=\ker(\varphi_{n,m})\cap \freeinvar_{n,m}^n$ 
we shall start with the relations in Proposition~\ref{prop:general} belonging to $\ker(\varphi_{n,m})$ and eliminate the variables $t(w)$ with $\deg(w)>n$. There is  one exception,   
we construct an element $\relthree$ in $\ker(\varphi_{n,n}^2)=\ker(\varphi)\cap \freeinvar_{n,n}^2$ by another method as follows: 
\begin{equation}\label{eq:relthree}\relthree:=\det\left(\begin{array}{ccccc} t(x_1^2)& t(x_1x_2)&\dots &t(x_1x_n) & t(x_1) 
\\ t(x_2x_1) &t(x_ 2^2) &\cdots &t(x_2 x_n)& t(x_2)\\
\vdots & \vdots&\ddots &\vdots &\vdots \\
t(x_nx_1)& t(x_nx_2)& \cdots & t(x_n^2) & t(x_n)\\
t(x_1)& t(x_2) & \cdots & t(x_n) & n\end{array}\right).
\end{equation}
 
\begin{proposition}\label{prop:gram} 
The element $J$ belongs to $\ker(\varphi_{n,n}^2)$, and $g\cdot J=\det^2(g)J$ for any $g\in\gl_n$. 
\end{proposition} 

\begin{proof} Applying $\varphi$ to the entries of the $(n+1)\times (n+1)$ matrix in (\ref{eq:relthree}) 
we get the matrix $X^T\cdot X$, where 
$$X=\left(\begin{array}{ccccc}x_{11}& x_{12}&\cdots & x_{1n}& 1\\
x_{21}&x_{22}&\cdots & x_{2n}& 1\\
\vdots &\vdots & \ddots& \vdots &\vdots\\
x_{n1}&x_{n2}&\cdots & x_{nn} &1\end{array}\right),$$ 
and $X^T$ denotes the transpose of $X$. Since $X$ has size $n\times (n+1)$, the rank of $X^TX$ is at most $n$, hence $\det(X^TX)=0$, showing that $J\in\ker(\varphi)$. 

To explain the second statement, let us describe the $\gl_m$-action on $\freeinvar=S(W_{n,m})$ more explicitly. 
First of all, $\gl_m$ acts by $\mc$-algebra automorphisms on the $m$-variable polynomial ring 
$\mc[x_1,\ldots,x_m]$. Namely, $g\in\gl_m$ sends the variable $x_i$ to the $i$th entry of 
the row vector $(x_1,x_2,\ldots,x_m)\cdot g$ (matrix multiplication). Note that the degree $d$ homogeneous component of $\mc[x_1,\ldots,x_m]$ is $\gl_m$-stable, and as a $\gl_m$-module, it is isomorphic to the $d$th symmetric tensor power of the natural $m$-dimensional representation of $\gl_m$ on the space $\mc^m$ of column vectors. 
Write $U$ for the $\mc$-linear span of $\monoms_m^n$ in $\mc[x_1,\ldots,x_m]$ (so $U$ is the sum of the homogeneous components of degree $1,2,\ldots,n$). 
It is easy to see that the $\mc$-linear map from $U\to R_{n,m}$ induced by  $w\mapsto [w]$ ($w\in\monoms_m^n$)  
is a $\gl_m$-module isomorphism, so we have $U\cong W_{n,m}$ as a $\gl_m$-module. 
This shows that the effect of $g\in\gl_m$ on a variable $t(w)$ of $\freeinvar_{n,m}^n$ is given by the formula $g\cdot t(w)=t(g\cdot w)$, 
where for an arbitrary polynomial $f=\sum_{w\in\monoms_m}a_ww\in\mc[x_1,\ldots,x_m]$, 
we shall mean by $t(f)$ the element $\sum_{w\in\monoms_m}a_wt(w)\in \freeinvar_{n,m}$. 

These considerations show that applying $g\in\gl_n$ to all entries of the matrix $Y$ in (\ref{eq:relthree}) 
we get the matrix $(\tilde g)^TY\tilde g$, where $\tilde g$ stands for the $(n+1)\times (n+1)$
block diagonal  matrix 
$\left(\begin{array}{cc}g&0\\0&1\end{array}\right)$. 
Therefore our statement follows by multiplicativity of the determinant. 
\end{proof}

The idea of applying the $\gl_m$-module structure in the study of generators and relations of rings of invariants $\mc[V^m]^G$ (where $G$ is a group of linear transformations on $V$) is well known, see for example section 5.2.7 in \cite{goodman-wallach}, or \cite{benanti-drensky} for a recent application. 
We collect some necessary facts on representations of $\gl_m$ on polynomial rings. 

The representation of $\gl_m$ on $S(W_{n,m})$ is a {\it polynomial representation} (cf.  \cite{macdonald}). Recall that polynomial $\gl_m$-modules are completely reducible. The isomorphism classes of irreducible polynomial representations are labeled by the set $\partitions_m$ of partitions with at most $m$ non-zero parts. By a partition $\lambda=(\lambda_1,\ldots,\lambda_m)\in\partitions_m$ we mean a decreasing sequence $\lambda_1\geq\lambda_2\geq\cdots\geq\lambda_m\geq 0$ of non-negative integers, and write $h(\lambda)$ for the number of non-zero parts of $\lambda$. 
Given $\lambda\in\partitions_m$ we denote by $V_{\lambda}$ a copy of 
an irreducible polynomial $\gl_m$-module labeled by $\lambda$. 
For example, $V_{(k)}$ is isomorphic to the degree $k$ homogeneous component of $\mc[x_1,\ldots,x_m]$, the $k$th symmetric tensor power of the natural $\gl_m$-module $\mc^m$, 
and $V_{(1^k)}:=V_{(1,\ldots,1)}\cong \bigwedge^k(\mc^m)$, the $k$th exterior power of 
$\mc^m$. We have the $\gl_m$-module isomorphism 
$$W_{n,m}\cong V_{(1)}\oplus V_{(2)}\oplus\cdots\oplus V_{(n)}.$$ 

Write $\unipotent_m$ for the subgroup of unipotent upper triangular matrices in $\gl_m$, and write 
$\torus\cong(\mc^{\times})^m=\mc^{\times}\times\cdots\times\mc^{\times}$ for the maximal torus consisting of diagonal matrices. Given a polynomial $\gl_m$-module $M$, we say that a non-zero $v\in M$ is a 
{\it highest weight vector of weight} $\lambda$ if $v$ is fixed by $\unipotent_m$, $\torus$ stabilizes 
$\mc v$ and acts on it by the weight $\lambda$, i.e.  
$\diag(z_1,\ldots,z_m)\cdot v=(z_1^{\lambda_1}\cdots z_m^{\lambda_m})v$. 
In this case the $\gl_m$-submodule generated by $v$ is isomorphic to $V_{\lambda}$. 
The irreducible $\gl_m$-module $V_{\lambda}$ contains a unique 
(up to non-zero scalar multiples) vector fixed by $\unipotent_m$, and (up to non-zero scalar multiples), it is the only vector in $V_{\lambda}$ on which $\torus$ acts by the weight $\lambda$. 

The action of $\torus$ defines a $\mz^m$-grading on any $\gl_m$-module $M$: 
$v\in M$ is {\it multihomogeneous of multidegree} $\alpha=(\alpha_1,\ldots,\alpha_m)$ if 
$\diag(z_1,\ldots,z_m)\cdot v=(z_1^{\alpha_1}\cdots z_m^{\alpha_m})v$ for all 
$\diag(z_1,\ldots,z_m)\in\torus$. 
In particular, $\mc[V^m]$, $R_{n,m}$, $\freeinvar_{n,m}$ become $\mz^m$-graded algebras this way, and the map $\varphi:\freeinvar_{n,m}\to R_{n,m}$ is multihomogeneous. 
The polarized power sum $[x_1^{\alpha_1}\cdots x_m^{\alpha_m}]$ is multihomogeneous of multidegree $(\alpha_1,\ldots,\alpha_m)$. 

The above $\mz^m$-grading is a refinement of the ususal $\mz$-grading on the polynomial algebra 
$\mc[V^m]$. Similarly, $R_{n,m}$ and $\freeinvar_{n,m}$ are graded algebras, the degree of a multihomogeneous element of multidegree $(\alpha_1,\ldots,\alpha_m)$ being 
$\sum_{j=1}^m\alpha_j$. 
Denote by $\freeinvar^{(+)}$ the sum of homogeneous components of positive degree in the graded algebra $\freeinvar=\freeinvar_{n,m}^n$. Then $\freeinvar^{(+)}$ is a maximal ideal and 
$\freeinvar/\freeinvar^{(+)}\cong\mc$. The ideal $\ker(\varphi)=\ker(\varphi_{n,m}^n)$ is homogeneous. 
By a {\it minimal system of generators of} $\ker(\varphi)$ we mean a set of homogeneous elements that constitutes an  irredundant generating system  of the ideal $\ker(\varphi)$. 
It is well known that a  subset 
$N\subset\ker(\varphi)$ of homogeneous elements is a minimal generating system of the ideal $\ker(\varphi)$ if and only if $N$ is a basis of a $\mc$-vector space direct complement of  
$\freeinvar^{(+)}\ker(\varphi)$ in $\ker(\varphi)$. 
This shows that although the minimal generating system  $N$ is not unique, for each degree the number of elements 
in $N$ of that degree is uniquely determined. Even more, we may assume that $N$ spans a 
$\gl_m$-submodule $\Span\{N\}$ of $\freeinvar$, and the $\gl_m$-module structure of 
$\Span\{N\}$ is uniquely determined by $\ker(\varphi)$. 
Indeed, note that an irreducible $\gl_m$-submodule of $\freeinvar$ isomorphic to $V_{\lambda}$ is contained  
in the homogeneous component of $\freeinvar$ of degree $|\lambda|=\sum_{j=1}^m\lambda_j$. 
Therefore we may take a $\gl_m$-module direct complement of $\freeinvar^{(+)}\ker(\varphi)$ in 
$\ker(\varphi)$, and a homogeneous $\mc$-basis of this complement is a minimal generating system of $\ker(\varphi)$ with the desired properties. 

\begin{definition}\label{def:mult} 
For $\lambda\in\partitions_m$ denote by $\mult_{n,m}(\lambda)$ the multiplicity of the irreducible $\gl_m$-module $V_{\lambda}$ as a summand in the factor $\gl_m$-module 
$\ker(\varphi_{n,m}^n)/(\freeinvar_{n,m}^n)^{(+)}\cdot \ker(\varphi_{n,m}^n)$, and for a partition 
$\lambda$ with more than $m$ non-zero parts set   
$\mult_{n,m}(\lambda)=0$.     \end{definition}

Next we recall that the multiplicities $\mult_{n,m}(\lambda)$ have only a mild dependence on the parameter $m$. 
Note that a highest weight vector $f$ of weight $\lambda$ in $\freeinvar_{n,m}$ is in particular multihomogeneous of multidegree $\lambda$, hence is 
contained in the subalgebra $\freeinvar_{n,h(\lambda)}$. 

\begin{lemma}\label{lemma:m1m2} 
Let $f$ be a multihomogeneous element of multidegree $\lambda$ in $\freeinvar_{n,h(\lambda)}$, 
and let $m_1\geq m_2\geq h(\lambda)$ be positive integers. 
Then $f$ is a highest weight vector for the action of $\gl_{m_1}$ on $\freeinvar_{n,m_1}$ if and only if 
it is a highest weight vector for the action of $\gl_{m_2}$ on $\freeinvar_{n,m_2}$. 
\end{lemma} 

\begin{proof} This is well known, and follows directly from the rule giving the action of the unipotent subgroup $\unipotent_{m_1}$ (resp. $\unipotent_{m_2}$) on $\freeinvar_{n,m_1}$ 
(resp. $\freeinvar_{n,m_2}$).
\end{proof} 

\begin{corollary}\label{cor:stable} 
Let $\lambda$ be a  partition. Then  we have 
$$\mult_{n,m}(\lambda)=\begin{cases}\mult_{n,h(\lambda)}(\lambda)\quad&\mbox{ if }m\geq h(\lambda);\\
0&\mbox{ if } m<h(\lambda).
\end{cases}$$
\end{corollary} 

\begin{proof} The case $m<h(\lambda)$ is trivial by definition of $\mult_{n,m}(\lambda)$. Suppose 
$m>h(\lambda)$. Denote by $K_{n,m}$ the kernel of $\varphi:\freeinvar_{n,m}^n\to R_{n,m}$, and write 
$K_{n,h(\lambda)}:=K_{n,m}\cap\freeinvar_{n,h(\lambda)}^n$ for the kernel of the restriction of $\varphi$ to 
$\freeinvar_{n,h(\lambda)}^n$. 
Let $N$ be a  $\gl_{h(\lambda)}$-module complement of 
$(\freeinvar_{n,h(\lambda)}^n)^{(+)}K_{n,h(\lambda)}$ in $K_{n,h(\lambda)}$. Decompose $N$ as a direct sum 
$\bigoplus N_i$ of irreducible $\gl_{h(\lambda)}$-modules, and take a highest weight vector $f_i$ in $N_i$ for all $i$. Then by Lemma~\ref{lemma:m1m2}, the $f_i$ are highest weight vectors for the action of $\gl_m$ on $\freeinvar_{n,m}^n$. Moreover, taking into account the $\mz^m$-grading, one can easily show that $N\cap(\freeinvar_{n,m}^n)^{(+)}K_{n,m}=0$, hence the $f_i$ are linearly independent 
modulo $(\freeinvar_{n,m}^n)^{(+)}K_{n,m}$. This shows the inequality 
$\mult_{n,h(\lambda)}(\lambda)\leq\mult_{n,m}(\lambda)$. The proof of the reverse inequality is similar. 
\end{proof}

To simplify notation, write 
$$\mult_n(\lambda):=\mult_{n,h(\lambda)}(\lambda).$$ 
The numbers $\mult_n(\lambda)$ (all but finitely many of them are zero) carry all the sensible numerical information on a minimal generating system of the kernel of $\varphi:\freeinvar_{n,m}^n\to R_{n,m}$ by 
Corollary~\ref{cor:stable}. 
Moreover, setting  
$$h(\ker(\varphi)):=\max\{h(\lambda)\mid \mult_n(\lambda)\neq 0\}$$ 
we have the following: 

\begin{corollary}\label{cor:h(kerfi)} 
The kernel of $\varphi:\freeinvar_{n,m}^n\to R_{n,m}$ is generated as a $\gl_m$-stable ideal by its intersection with $\freeinvar_{n,h(\ker(\varphi))}^n$. 
\end{corollary} 

\begin{proof} Let $\oplus_i M_i$ be a $\gl_m$-module direct complement of 
$\freeinvar^{(+)}\ker(\varphi)$ in $\ker(\varphi)$, where the $M_i$ are irreducible $\gl_m$-modules. 
Let $f_i$ be a highest weight vector in $M_i$. Then all the $f_i$ belong to $\freeinvar_{n,h(\ker(\varphi))}^n$, and the $\gl_m$-module  $\oplus_i M_i$ generated by the set 
$\{f_i\}$ generates $\ker(\varphi)$ as an ideal. 
\end{proof} 


\section{Main results}\label{sec:main} 

First specialize to $n=3$. 
The fundamental elements  $\Psi (w_1, w_2, w_3, w_4):=\Psi_4(w_1,w_2,w_3,w_4)$ of 
$\ker(\varphi_{3,m})$ defined in (\ref{eq:fundrel}) (here $w_1,w_2,w_3,w_4\in\monoms_m$) 
take the form 
\begin{align*}\Psi (w_1, w_2, w_3, w_4) &=-6t(w_1w_2w_3w_4) \\
&+2 t(w_1w_2w_3)t(w_4)+2t(w_1w_2w_4)t(w_3)+2t(w_1w_3w_4)t(w_2)+2t(w_2w_3w_4)t(w_1) \\
&+t(w_1w_2)t(w_3w_4)+t(w_1w_3)t(w_2w_4)+t(w_1w_4)t(w_2w_3)\\
&-t(w_1w_2)t(w_3)t(w_4)-t(w_1w_3)t(w_2)t(w_4)-t(w_1w_4)t(w_2)t(w_3)\\
&-t(w_2w_3)t(w_1)t(w_4)-t(w_2w_4)t(w_1)t(w_3)-t(w_3w_4)t(w_1)t(w_2)\\
&+t(w_1)t(w_2)t(w_3)t(w_4).
\end{align*}

Next we define an element $\relone_{3,2}$ and an element $\reltwo_{4,2}$ in $\freeinvar_{3,2}$; 
they have multidegree $(3,2)$, and $(4,2)$, respectively. 
To simplify notation, we write $x,y$ instead of $x_1,x_2$, so $\monoms_2$ consists of monomials in the commuting indeterminates $x,y$. 
 $$\relone_{3,2}:= \frac 12\bigl(3\Psi(xy,x,x,y)-3\Psi(x,x,x,y^2)
+\Psi(x,x,x,y)t(y)-\Psi(x,x,y,y)t(x)\bigr)$$
\begin{align*}\reltwo_{4,2}&:=3\Psi(xy,xy,x,x)-3\Psi(x,x,x,xy^2)
+2\Psi(x,x,x,y)t(xy)
\\&-\Psi(x,x,y,y)t(x^2)-\Psi(x,x,x,y^2)t(x)
\end{align*}
The elements $\relone_{3,2}$ and $\reltwo_{4,2}$ are both contained in 
$\freeinvar_{3,2}^3$ (i.e. they do not involve variables $t(w)$ with $\deg(w)>3$). 
Indeed, direct calculation shows that 
\begin{align*}
\relone_{3,2}&=6t(x^2y)t(xy)-3t(xy^2)t(x^2)-2t(x^2y)t(x)t(y)
\\&+t(xy^2)t(x)^2-4t(xy)^2t(x)+2t(xy)t(x)^2t(y)
-3t(x^3)t(y^2)\\&+4t(x^2)t(x)t(y^2)-t(x)^3t(y^2)
+t(x^3)t(y)^2-t(x^2)t(x)t(y)^2
\end{align*}
and 
\begin{align*}\reltwo_{4,2}&=6t(x^2y)^2+t(xy)^2t(x^2)-3t(xy)^2t(x)^2-6t(x^3)t(xy^2)
\\&+2t(x^2)t(xy^2)t(x)+4t(x^3)t(xy)t(y)
\\&-2t(x^2)t(xy)t(x)t(y)+2t(xy)t(x)^3t(y)-4t(x^2y)t(x^2)t(y)
\\&-t(x^2)^2t(y^2)+t(x^2)^2t(y)^2+4t(x^2)t(x)^2t(y^2)
\\&-t(x^2)t(x)^2t(y)^2-t(x)^4t(y^2)-2t(x^3)t(x)t(y^2)
\end{align*}

Finally, denote by $\relthree_{2,2,2}$ the element for $n=3$ defined by (\ref{eq:relthree}) in general. 
Clearly, $\relthree_{2,2,2}$ belongs to $\freeinvar_{3,3}^2$ and has multidegree $(2,2,2)$. 

\begin{theorem}\label{thm:main} We have 
$$\mult_3(\lambda)=\begin{cases} 1 &\mbox{ for }\lambda=(3,2), \quad \lambda=(4,2),  \mbox{ and }\quad \lambda=(2,2,2);\\   
0 &\mbox{ for all other }\lambda. \end{cases} $$ 
For $m\geq 2$ the elements $\relone_{3,2}$, $\reltwo_{4,2}\in\freeinvar_{3,m}^3$   generate irreducible $\gl_m$-submodules $N_{(3,2)}^m\cong V_{(3,2)}$, $N_{(4,2)}^m\cong V_{(4,2)}$ in $\ker(\varphi)$, 
and for $m\geq 3$ the element $\relthree_{2,2,2}\in\freeinvar_{3,m}^3$ generates an irreducible $\gl_m$-submodule  $N_{(2,2,2)}^m\cong V_{(2,2,2)}$ in $\ker(\varphi)$. 
Furthermore, choose arbitrary $\mc$-bases $\basis_{(3,2)}^m$, $\basis_{(4,2)}^m$,  $\basis_{(2,2,2)}^m$  in 
$N_{(3,2)}^m$, $N_{(4,2)}^m$, $N_{(2,2,2)}^m$. Set 
$\basis^m:=\basis_{(3,2)}^m\cup\basis_{(4,2)}^m\cup\basis_{(2,2,2)}^m$ when $m\geq 3$ and 
$\basis^2:=\basis_{(3,2)}^2\cup\basis_{(4,2)}^2$. Then $\basis^m$  
is a minimal generating system of the kernel of the surjection $\varphi:\freeinvar_{3,m}^3\to R_{3,m}$ for any $m\geq 2$. 
\end{theorem}
\goodbreak 

In classical language (see for example \cite{weyl}), the elements in the $\gl_m$-module generated by $f\in\ker(\varphi)$ are called the {\it polarizations of} $f$.  So Theorem~\ref{thm:main} can be paraphrased 
as follows: the polarizations of $\relone_{3,2}$, $\reltwo_{4,2}$, $\relthree_{2,2,2}$ minimally generate the ideal of relations among the polarized power sums of degree at most three in dimension  three.   

For an arbitrary $n$ we show that the ideal of relations between the polarized power sums can not be generated in degree strictly less than $2n$:  

\begin{theorem}\label{thm:lowerbound} 
Suppose $m\geq n$. The element $\relone\in\ker(\varphi_{n,m}^n)$ given in (\ref{eq:relthree}) does not belong to 
the ideal $(\freeinvar_{n,m}^n)^{(+)}\cdot\ker(\varphi_{n,m}^n)$.  
In particular, denoting  by $2^n:=(2,\ldots,2)\in\partitions_n$ the partition with $n$ non-zero parts, all equal to $2$, we have $\mult_n(2^n)=1$. 
\end{theorem}

Denote by $\beta(n,m)$ the largest degree of an element in a minimal generating system of the ideal $\ker(\varphi_{n,m}^n)$ of relations between the polarized power sums. 
We summarize our present knowledge of $\beta(n,m)$.  
The general upper bound 
$$\beta(n,m)\leq n^2-n+2$$ 
is pointed out in \cite{domokos}. 
By Theorem~\ref{thm:lowerbound} we have the general lower bound 
$$\beta(n,m)\geq 2n \mbox{ for }m\geq n.$$  
We have $\beta(2,m)=4$ for all $m\geq 2$ (see for example \cite{domokos}), and by 
Theorem~\ref{thm:main} we have 
$$\beta(3,m)=6  \mbox{ for all }m\geq 2.$$  
Note that both for $n=2$ and $n=3$ the exact value of $\beta(n,m)$ agrees with the general lower bound $2n$ established here. 
Moreover, both for $n=2$ and $n=3$ we have the equality $h(\ker(\varphi^n_{n,m}))=n$ when 
$m\geq n$. 


\section{Reduction to $m=4$} \label{sec:reductionm=4}

\begin{proposition}\label{prop:d/2} 
If $V_{\lambda}$ occurs as an irreducible $\gl_m$-module summand in the degree $d$ homogeneous component of 
$\freeinvar_{n,m}^n$, then $h(\lambda)\leq \frac{d+1}{2}$. 
\end{proposition}

\begin{proof} Denote by $S^k(M)$ the $k$th symmetric tensor power of the $\gl_m$-module $M$. 
The degree $d$ homogeneous component of $\freeinvar_{n,m}^n$ is isomorphic as a $\gl_m$-module to 
\begin{equation}\label{eq:symmetricpowers}
\bigoplus_{d_1+2d_2+3d_3+\cdots+nd_n=d}S^{d_1}(V_{(1)})\otimes S^{d_2}(V_{(2)})\otimes\cdots\otimes S^{d_n}(V_{(n)}).\end{equation}
Note that $S^{d_1}(V_{(1)})\cong V_{(d_1)}$, and for $i\geq 2$,  $S^{d_i}(V_{(i)})$ is a $\gl_m$-submodule of $V_{(i)}\otimes\cdots\otimes V_{(i)}$ ($d_i$ tensor factors), which involves only summands $V_{\lambda}$ with $h(\lambda)\leq d_i$ by Pieri's rule (I.5.16 in \cite{macdonald}). 
One concludes by the Littlewood-Richardson rule (I.9.2  in \cite{macdonald} that for the irreducible constituents $V_{\lambda}$ of (\ref{eq:symmetricpowers}) we have $h(\lambda)\leq 1+d_2+d_3+\cdots +d_n\leq \frac{d+1}{2}$  
(the latter inequality follows from $d= \sum_{i=1}^n id_i$).  
\end{proof} 

\begin{proposition}\label{prop:n^2-n+2/2} 
We have the inequality $h(\ker(\varphi))\leq (n^2-n+2)/2$. Consequently,  
the kernel of $\varphi:\freeinvar_{n,m}^n\to R_{n,m}$ is generated as a $\gl_m$-stable ideal by its intersection with $\freeinvar_{n,(n^2-n+2)/2}^n$. 
\end{proposition}

\begin{proof} We know from Proposition~\ref{prop:general} that $\ker(\varphi)$ is generated as an ideal by the sum $M$ of its homogeneous components of degree $\leq n^2-n+2$. Decompose $M$ as the direct sum $\bigoplus_iM_i$ of irreducible $\gl_m$-modules. By Proposition~\ref{prop:d/2},  $M_i\cong V_{\lambda_i}$ for some $\lambda_i\in\partitions_m$ with $h(\lambda_i)\leq (n^2-n+2)/2$. 
Since $M$ contains a $\gl_m$-module direct complement of $(\freeinvar_{n,m}^n)^{(+)}\ker(\varphi)$  in  
the $\gl_m$-module $\ker(\varphi)$, it follows that  $\mult_{n,m}(\lambda)=0$ when  $h(\lambda)>\frac{n^2-n+2}{2}$. This shows the inequality $h(\ker(\varphi))\leq n^2-n+2$, implying by 
Corollary~\ref{cor:h(kerfi)} the second statement. 
\end{proof} 

For $n=3$ we have $\frac{n^2-n+2}{2}=4$, hence by Proposition~\ref{prop:n^2-n+2/2} 
it is sufficient to prove Theorem~\ref{thm:main} in the special case $m=4$. 


\section{Minimality}\label{sec:minimality}

Throughout this section we assume $n=3$. First we determine the $\gl_m$-module structure of the  
kernel $K_{3,m}$ of $\varphi:\freeinvar_{3,m}^3\to R_{3,m}$ up to degree $6$. Denote by 
$K_{3,m}^{(d)}$ the degree $d$ homogeneous component of $K_{3,m}$. Note that similarly to 
Corollary~\ref{cor:stable} one has that the multiplicity of $V_{\lambda}$ as a summand in the $\gl_{h(\lambda)}$-module $K_{3,h(\lambda)}$ is the same as the multiplicity of $V_{\lambda}$ in the 
$\gl_m$-module $K_{3,m}$ for an arbitrary $m\geq h(\lambda)$.

\begin{proposition}\label{prop:kernelglstructure}
We have $K_{3,m}^{(d)}=0$ for $d\leq 4$, and for $d=5,6$ the following $\gl_m$-module isomorphisms hold: 
$$K_{3,m}^{(5)}\cong V_{(3,2)} \quad \mbox{ for }m\geq 2;$$
$$K_{3,m}^{(6)}\cong 2\cdot V_{(4,2)}+V_{(3,3)}+V_{(3,2,1)}+V_{(2,2,2)}\quad \mbox{ for }m\geq 3.$$
\end{proposition} 

\begin{proof} The fact that $K_{3,m}^{(d)}=0$ for $d\leq 4$ follows for example from Proposition~\ref{prop:general}. 
Denote by $(\freeinvar_{3,m}^3)^{(d)}$ and $R_{3,m}^{(d)}$ the degree $d$ homogeneous component of $\freeinvar_{3,m}^3$ and $R_{3,m}$. 
By formula (\ref{eq:symmetricpowers}) we have 
\begin{align*}(\freeinvar_{3,m}^3)^{(5)}&\cong S^5(V_{(1)})+S^3(V_{(1)})\otimes V_{(2)}+S^2(V_{(1)})\otimes V_{(3)}+
V_{(2)}\otimes V_{(3)}+V_{(1)}\otimes S^2(V_{(2)})\\
&\cong V_{(5)}+3\cdot V_{(3)}\otimes V_{(2)}+V_{(1)}\otimes S^2(V_{(2)}), 
\end{align*}
whereas 
\begin{align*}(\freeinvar_{3,m}^3)^{(6)}&\cong 
V_{(6)}+V_{(4)}\otimes V_{(2)}+V_{(3)}\otimes V_{(3)}+V_{(2)}\otimes S^2(V_{(2)})
\\ &+V_{(1)}\otimes V_{(2)}\otimes V_{(3)}+ S^3(V_{(2)})+S^2(V_{(3)}).
\end{align*} 
By Pieri's rule and some known plethysm formulae (see Section I.8 in \cite{macdonald}) one derives 
$$(\freeinvar_{3,m}^3)^{(5)}\cong 5\cdot V_{(5)}+4\cdot V_{(4,1)}+4\cdot V_{(3,2)}+V_{(2,2,1)}$$ 
and 
$$(\freeinvar_{3,m}^3)^{(6)}\cong 7\cdot V_{(6)}+5\cdot V_{(5,1)}+8\cdot V_{(4,2)}+
V_{(4,1,1)}+2\cdot V_{(3,3)}+2\cdot V_{(3,2,1)}+2\cdot V_{(2,2,2)}.$$
To determine the $\gl_m$-module structure of $R_{3,m}$ we start from 
the action of $\gl_3\times\gl_m$ on $\mc[V^m]=\mc[x_{ij}\mid i=1,2,3,\quad j=1,\ldots,m]$ by $\mc$-algebra automorphisms given on the generators as follows: 
$(g,h)\in\gl_3\times\gl_m$ sends $x_{ij}$ to the $(i,j)$-entry of the matrix 
$g^T(x_{ij})_{i=1,2,3}^{j=1,\ldots,m}h$. Cauchy's formula (I.4.3 in \cite{macdonald}) tells us the $\gl_3\times\gl_m$-module structure of $\mc[V^m]$: 
$$\mc[V^m]\cong \bigoplus_{\lambda\in\partitions_{\min\{3,m\}}}V_{\lambda}\otimes V_{\lambda}$$
Consequently, 
the multiplicity of $V_{\lambda}$ in $R_{3,m}$ equals $\dim_{\mc}(V_{\lambda}^{S_3})$, where we identify $S_3$ with the subgroup of permutation matrices in $\gl_3$, and we write $V_{\lambda}^{S_3}$ for the subspace of $S_3$-fixed points in the $\gl_3$-module $V_{\lambda}$. 
By the Jacobi-Trudi Formula (I.3.4 in \cite{macdonald}), the character of an element $g\in S_3$ on $V_{\lambda}$ (where $h(\lambda)\leq 3$) 
equals the determinant of the $3\times 3$ matrix whose $(i,j)$-entry is $h_{\lambda_i-i+j}(z_1,z_2,z_3)$, 
where $h_k(z_1,z_2,z_3)$ is the $k$th complete symmetric polynomial in the eigenvalues 
$z_1,z_2,z_3$ of $g\in S_3 <\gl_3$.  On the other hand, if $g\in S_3<\gl_3$ has eigenvalues 
$z_1,z_2,z_3$, then $h_k(z_1,z_2,z_3)$ equals the number of monomials of degree $k$ in the variables $x_1,x_2,x_3$ fixed by $g$ (where $S_3$ acts by permuting the variables). 
Based on this one can quickly compute the $\gl_m$-module structure of $R_{3,m}$, and gets 
$$R_{3,m}^{(5)}\cong 5\cdot V_{(5)}+4\cdot V_{(4,1)}+3\cdot V_{(3,2)}+V_{(2,2,1)}$$ 
and 
$$R_{3,m}^{(6)}\cong 7\cdot V_{(6)}+5\cdot V_{(5,1)}+6\cdot V_{(4,2)}+V_{(4,1,1)}+V_{(3,3)}
+V_{(3,2,1)}+V_{(2,2,2)}.$$ 

Since the multiplicity of $V_{\lambda}$ in $K_{3,m}$ equals the difference of the multiplicities of 
$V_{\lambda}$ in $\freeinvar_{3,m}^3$ and in $R_{3,m}$, the statement follows. 
\end{proof}

\begin{proposition}\label{prop:highestweight} 
$\relone_{3,2}$, $\reltwo_{4,2}$, $\relthree_{2,2,2}$ are highest weight vectors for the action of $\gl_m$ on $K_{3,m}$, and none of them is contained in the ideal $(\freeinvar_{3,m}^3)^{(+)}K_{3,m}$. 
\end{proposition} 

\begin{proof} Recall that the multidegree of any element of $V_{\lambda}$ is lexicographically smaller than $\lambda$. Therefore 
Proposition~\ref{prop:kernelglstructure} shows that the multihomogeneous component of multidegree $(3,2)$ in $K_{3,m}$ is one-dimensional, and its non-zero elements are the highest weight vectors of the summand $V_{(3,2)}$. On the other hand, $J_{3,2}$ belongs to $K_{3,m}$ and has multidegree $(3,2)$, so it is a highest weight vector. Moreover, since $K_{3,m}$ does not contain elements of degree less than five, we conclude that $J_{3,2}$ is not contained in $(\freeinvar_{3,m}^3)^{(+)}K_{3,m}$. 

Similarly, an inspection of the decomposition of $K_{3,m}^{(6)}$ given in Proposition~\ref{prop:kernelglstructure} shows that  the multihomogeneous component of multidegree $(4,2)$ is 
two-dimensional, and all its non-zero elements are highest weight vectors generating a submodule isomorphic to $V_{(4,2)}$. Consequently, $\reltwo_{4,2}$ is a highest weight vector. 

By Proposition~\ref{prop:gram} we know that $g\cdot J_{2,2,2}=\det^2(g) J_{2,2,2}$ for any $g\in \gl_3$, 
hence in the special case $m=3$,  $J_{2,2,2}$ spans a one-dimensional $\gl_3$-submodule isomorphic 
to $V_{(2,2,2)}$. Consequently, $J_{2,2,2}$ is a highest weight vector for any $m\geq 3$ by Lemma~\ref{lemma:m1m2}. 

Since the  minimal degree of an element of $K_{3,m}$ is $5$, we have 
$$K_{3,m}^{(6)}\cap (\freeinvar_{3,m}^3)^{(+)}K_{3,m}
=\sum_{j=1}^mt(x_j)K_{3,m}^{(5)}.$$ 
Note that $J_{4,2}$ contains the term 
$6t(x_1^2x_2)^2$ and $J_{2,2,2}$ contains the term  
$3t(x_1x_2)t(x_1x_3)t(x_2x_3)$. We conclude that none of them is contained in 
 $(\freeinvar_{3,m}^3)^{(+)}K_{3,m}$. 
\end{proof} 

Denote by $N_{(3,2)}^m,N_{(4,2)}^m,N_{(2,2,2)}^m$ the $\gl_m$-submodules in $\freeinvar_{3,m}^3$ generated by 
$J_{3,2},J_{4,2},J_{2,2,2}$. 

\begin{corollary}\label{cor:N1N2N3} 
$N_{(3,2)}^m$, $N_{(4,2)}^m$, and $N_{(2,2,2)}^m$ are irreducible $\gl_m$-submodules of $K_{3,m}$ isomorphic to $V_{(3,2)}$, 
$V_{(4,2)}$, and $V_{(2,2,2)}$. Moreover, the intersection of 
$N_{(3,2)}^m + N_{(4,2)}^m + N_{(2,2,2)}^m$ and the ideal $(\freeinvar_{3,m}^3)^{(+)}K_{3,m}$ is zero. 
\end{corollary}

\begin{proof} Taking into account the multidegrees of $\relone_{3,2},\reltwo_{4,2},\relthree_{2,2,2}$, the first statement immediately follows from Proposition~\ref{prop:highestweight}. 
Moreover, $N_{(3,2)}^m,N_{(4,2)}^m,N_{(2,2,2)}^m$ are pairwise non-isomorphic irreducible $\gl_m$-modules and none of them is contained in the $\gl_m$-module $(\freeinvar_{3,m}^3)^{(+)}K_{3,m}$ (again by Proposition~\ref{prop:highestweight}), hence their sum $N_{(3,2)}^m+N_{(4,2)}^m+N_{(2,2,2)}^m$ is disjoint from  $(\freeinvar_{3,m}^3)^{(+)}K_{3,m}$ by basic principles of representation theory.  
\end{proof}

Choose arbitrary $\mc$-bases $\basis_{(3,2)}^m$, $\basis_{(4,2)}^m$, and $\basis_{(2,2,2)}^m$  in 
$N_{(3,2)}^m$, $N_{(4,2)}^m$, and $N_{(2,2,2)}^m$,  and set 
$\basis^m:=\basis_{(3,2)}^m\cup\basis_{(4,2)}^m\cup\basis_{(2,2,2)}^m$. Then by Corollary~\ref{cor:N1N2N3}, $\basis^m$ can be extended to a minimal system of generators of the ideal $K_{3,m}$. 
To prove that $\basis^m$ is actually a minimal system of generators of $K_{3,m}$, it is sufficient to show that the ideal $K_{3,m}$ can be generated by $|\basis^m|$ elements. 

Recall that the dimension of the $\gl_m$-module $V_{\lambda}$ (where $\lambda\in\partitions_m$)  equals 
$$d_{\lambda}(m):=\prod_{1\leq i<j\leq m}\frac{\lambda_i-\lambda_j+j-i}{j-i} $$ 
by the Weyl dimension formula (see for example (7.1.17) in \cite{goodman-wallach}), and so 
$$|\basis^m|=d_{(3,2)}(m)+d_{(4,2)}(m)+d_{(2,2,2)}(m).$$


\section{Hironaka decomposition} \label{sec:hironaka}

It is well known that 
$$\primary:=\{[x_j],[x_j^2],[x_j^3]\mid j=1,\ldots,m\}$$ 
is a {\it homogeneous system of parameters} in $R_{3,m}$. 
Write $\mc[\primary]$ for the $\mc$-subalgebra of $R$ generated by $\primary$. 
It is a polynomial ring in the $3m$ generators, and $R$ is a finitely generated $\mc[\primary]$-module. 
Moreover, since $R$ is Cohen-Macaulay, it is a free $\mc[\primary]$-module. A set $\secondary\subset R_{3,m}$ of homogeneous elements 
constitutes a free $\mc[\primary]$-module generating system of $R$ if and only if the image of $\secondary$ is a $\mc$-vector space basis of the factor algebra $R/(\primary)$ of $R$ modulo the ideal $(\primary)$ generated by $\primary$. The elements of $\primary$ (respectively $\secondary$) are refered to as the 
{\it primary} (respectively {\it secondary}) generators of $R_{3,m}$, and 
\begin{equation}\label{eq:hironaka}
R_{3,m}=\bigoplus_{s\in \secondary}\mc[\primary]\cdot s\end{equation} 
the {\it Hironaka decomposition} of $R_{3,m}$. 
Write $\mingen:=\{[x_1^{\alpha_1}\cdots x_m^{\alpha_m}]\mid 
\alpha_1+\cdots+\alpha_m\leq 3\}$ 
for the chosen minimal $\mc$-algebra generating system of $R_{3,m}$. We have $\mingen\supset \primary$, 
and we may assume that $\secondary$ consists of products of powers of the elements of $\mingen$ 
(in particular, then $\secondary$ consists of multihomogeneous elements, and the empty product $1\in \secondary$). 
Recall that the Hilbert series of an $\mn_0^m$-graded vector space 
$A:=\bigoplus_{\alpha}A^{\alpha}$ with $\dim_{\mc}(A^{\alpha})<\infty$ is the formal power series in $\mz[[t_1,\ldots,t_m]]$ 
defined by 
$$H(A;t_1,\ldots,t_m):=\sum_{\alpha=(\alpha_1,\ldots,\alpha_m)}
\dim_{\mc}(A^{\alpha})t_1^{\alpha_1}\cdots t_m^{\alpha_m}.$$ 
It follows from (\ref{eq:hironaka}) that 
\begin{equation} \label{eq:hironakahilbert} 
H(R_{3,m};t_1,\ldots,t_m)=\frac{H(\Span_{\mc}(\secondary); t_1,\ldots,t_m)}
{\prod_{j=1}^m(1-t_j)(1-t_j^2)(1-t_j^3)}
\end{equation}  
where $\Span_{\mc}(\secondary)$ is the $\mc$-subspace in $R$ spanned by $\secondary$ 
(since $\secondary$ consists of multihomogeneous elements, it is $\mn_0^m$-graded). 
On the other hand, the Hilbert series of $R$ can be explicitly calculated (see Section~\ref{sec:hilbertseries}), and from this we know the number of elements of $\secondary$ having multidegree $\alpha$ for each $\alpha$. 

The following two statements provide our basis to find a complete system of relations. 

\begin{lemma} \label{lemma:congruences} 
Fix a positive integer $d$, and let $\trysecondary$ be a finite set of monomials in the elements of $Q$, each element of $\trysecondary$ having degree at most $d$,  and suppose that $\trysecondary$ satisfies the following: 
\begin{itemize}
\item[(i)] $1\in \trysecondary$, and $\mingen\setminus \primary\subseteq \trysecondary$; 
\item[(ii)] For each $e=1,\ldots,d$, the number of degree $e$ elements in $\trysecondary$ equals the number of degree $e$ elements in a system of secondary generators of $R_{3,m}$. 
\item[(iii)] For any $s\in \trysecondary$ and $q\in \mingen\setminus \primary$ with 
$\deg(s\cdot q)\leq d$ there exist scalars $\gamma_a\in\mc$ ($a\in \trysecondary$, $\deg(a)=\deg(sq)$) 
with 
\begin{equation}\label{eq:sq}
s\cdot q-\sum_{\deg(a)= \deg(sq)}\gamma_aa\in (\primary).
\end{equation} 
\end{itemize} 
Then $\trysecondary$ can be extended to a system $\secondary$ of secondary generators of $R_{3,m}$ such that 
$\trysecondary$ coincides with the subset of degree $\leq d$ elements of $\secondary$. 
\end{lemma}

\begin{proof} A straightforward induction on the degree. 
\end{proof} 

We shall use the following notation: for $f_1,f_2\in R$, we write $f_1\equiv f_2$ if 
$f_1-f_2\in (\primary)$. For example, (\ref{eq:sq}) reads as 
$$s\cdot q\equiv \sum_{\deg(a)=\deg(sq)}\gamma_aa,$$ 
and it means that there exists a multihomogeneous element $r_s^q$ in the kernel $K_{3,m}$ of the surjection 
$\varphi:\freeinvar_{3,m}^3\to R_{3,m}$ such that 
\begin{equation}\label{eq:rsq} 
r_s^q-s^\star\cdot q^\star+\sum_{\deg(a)=\deg(sq)}\gamma_a a^\star
\ \in \ \sum_{j=1}^m\sum_{k=1}^3\freeinvar_{3,m}^3\cdot t(x_j^k)
\end{equation} 
where given a product $c=[w_1]\cdots [w_l]$ of the generators $[w_i]\in \mingen$ we write 
$c^\star:=t(w_1)\cdots t(w_l)\in\freeinvar$. 
(Of course, $r_s^q$ is not unique, it is determined modulo the 
intersection of  $K_{3,m}$ and the ideal on the right hand side of (\ref{eq:rsq}).) 

\begin{lemma}\label{lemma:rsq} Suppose that the assumptions of Lemma~\ref{lemma:congruences} 
hold. For all $s\in S$, $q\in \mingen\setminus \primary$ with $\deg(qs)\leq d$ choose an element $r_s^q$ satisfying (\ref{eq:rsq}).  
Then the ideal generated by the $r_s^q$ contains all homogeneous components of $K_{3,m}$ up to degree $d$. 
\end{lemma} 

\begin{proof} Straightforward. 
\end{proof} 

We shall use also the special case $n=3$ of Lemma 6.1 from \cite{domokos}: 

\begin{lemma}\label{lemma:x^nequiv0} 
If the monomial $x_1^{\alpha_1}\cdots x_m^{\alpha_m}\in\monoms_m$ has degree at least $3$ in one of the variables $x_1,\ldots,x_m$, and $\alpha_1+\cdots+\alpha_m\geq 4$, then 
$[x_1^{\alpha_1}\cdots x_m^{\alpha_m}]\equiv 0$. 
\end{lemma} 


\section{Hilbert series} \label{sec:hilbertseries}

In this section we express the Hilbert series of $R_{3,m}$ in a form that is practical to evaluate for small $m$. 
The symmetric group $S_3$ has three irreducible complex characters: $\chi_0$, the trivial character, 
$\chi_1$, the character of the $2$-dimensional irreducible representation, and $\chi_2$, the sign character. Denote by $C$ the {\it character ring} of $S_3$; i.e. $C$ is the subring of the algebra of central functions on $S_3$ generated by $\chi_0,\chi_1,\chi_2$. It is a free $\mz$-module spanned by $\chi_0,\chi_1,\chi_2$. The multiplication in $C$ is given as follows: $\chi_0$ is the identity element, 
$\chi_1^2=\chi_0+\chi_1+\chi_2$, $\chi_2^2=\chi_0$, $\chi_1\cdot \chi_2=\chi_1$. 
For a graded $S_3$-module $A:=\bigoplus_{k=0}^{\infty}A^{(k)}$ we set 
$H_{\chi_i}(A;t):=\sum_{k=0}^{\infty}\mult_{\chi_i}(A^{(k)})t^k$, 
where  $\mult_{\chi_i}(A^{(k)})$ denotes the multiplicity of the irreducible representation with character $\chi_i$ as a summand of $A^{(k)}$. Moreover, 
set 
$$H_{S_3}(A;t):=\sum_{i=0}^2\chi_i H_{\chi_i}(A;t)\in C[[t]].$$  
One defines the Hilbert series of a multigraded $S_3$-module in a similar way. 
Clearly, the Hilbert series of the multigraded vector space $R$ coincides with the coefficient of 
$\chi_0$ in 
$$H_{S_3}(\mc[V^m];t_1,\ldots,t_m)\in  \sum_{i=0}^2 \mz[[t_1,\ldots,t_m]]\chi_i.$$ 
We have the isomorphism 
$\mc[V^m]\cong\mc[V]\otimes\cdots\otimes\mc[V]$, hence 
$$H_{S_3}(\mc[V^m];t_1,\ldots,t_m)=\prod_{j=1}^mH_{S_3}(\mc[V];t_j)$$
where multiplication is undestood in the ring of formal power series $C[[t_1,\ldots,t_m]]$ with coefficients in the character ring $C$ of $S_3$. 
It is well known that 
$$H_{S_3}(V;t)=\frac{\chi_0+(t+t^2)\chi_1+t^3\chi_2}{(1-t)(1-t^2)(1-t^3)}.$$  
Taking into account (\ref{eq:hironakahilbert}) we conclude that the 
Hilbert series of a system of multihomogeneous secondary generators $\secondary$ (defined in Section~\ref{sec:hironaka}) equals the coefficient of $\chi_0$ in 
\begin{equation}\label{eq:hilbS} \prod_{j=1}^m
(\chi_0+(t_j+t_j^2)\chi_1+t_j^3\chi_2)\in \sum_{i=0}^2 \mz[[t_1,\ldots,t_m]]\chi_i.
\end{equation} 


\section{The cases $m=2,3,4$}\label{sec:m=2,3,4} 

In this section we prove that the kernel $K_{3,4}$ of the surjection 
$\varphi:\freeinvar_{3,4}^3\to R_{3,4}$ can be generated by 
$$d_{(3,2)}(4)+d_{(4,2)}(4)+d_{(2,2,2)}(4)=60+126+10=196$$ 
elements. By the concluding remarks in Section~\ref{sec:minimality}, it follows that Theorem~\ref{thm:main} holds in the special case $m=4$. This finishes the proof of Theorem~\ref{thm:main} for arbitrary $m$ by the concluding remark of Section~\ref{sec:reductionm=4}. 

To simplify notation, we shall write $x,y,z,w$ instead of $x_1,x_2,x_3,x_4$. 


\subsection{The case $m=2$}\label{subsec:m=2} 

(This case is sketched in \cite{domokos}.) 
By (\ref{eq:hilbS}) we have 
$$H(\secondary;t,u)=1+tu+t^2u+tu^2+t^2u^2+t^3u^3.$$ 
Set 
$$S:=\{1, [xy], [x^2y],[xy^2], [xy]^2,[x^2y][xy^2] \}.$$ 
The equality $\varphi(\Psi(x,x,y,y))=0$ yields the congruence 
\begin{equation}\label{eq:x^2y^2}[x^2y^2]\equiv \frac 13[xy]^2\end{equation}
and the equality $\varphi(\Psi(xy,x,x,y))=0$ yields 
$$6[xy\cdot x\cdot x\cdot y]\equiv 4[x^2y][xy].$$ 
It follows by Lemma~\ref{lemma:x^nequiv0} that 
\begin{equation}\label{eq:[x^2y][xy]}
[x^2y][xy]\equiv 0.
\end{equation}
As explained before Lemma~\ref{lemma:rsq}, to the congruence (\ref{eq:[x^2y][xy]}) there belongs an element $r_{[x^2y]}^{[xy]}\in K_{3,2}$. 
By symmetry in $x$ and $y$, we have also the congruence and the corresponding relation: 
$$[xy^2][xy]\equiv 0\quad \mbox{implied by} \quad r_{[xy^2]}^{[xy]}\in K_{3,2}.$$
We have the congruence 
$$6[xy\cdot xy\cdot x\cdot x]\equiv 4[x^3y][xy]+2[x^2y]^2$$ 
(obtained by substituting the factors of $xy\cdot xy\cdot x\cdot x$ on the left hand side into $\Psi$), 
yielding by Lemma~\ref{lemma:x^nequiv0} 
$$[x^2y]^2\equiv 0\quad\mbox{and}\quad r_{[x^2y]}^{[x^2y]}\in K_{3,2}.$$
By symmetry, $[xy^2]^2\equiv 0$ from $r_{[xy^2]}^{[xy^2]}\in K_{3,2}$. 
Finally, we have 
$$6[xy\cdot xy\cdot x\cdot y]\equiv 5[x^2y^2][xy]+2[x^2y][xy^2]-[xy]^3$$
and taking into account Lemma~\ref{lemma:x^nequiv0} and (\ref{eq:x^2y^2}) we get 
$$[xy]^3\equiv -3[x^2y][xy^2]\quad \mbox{and}\quad r_{[xy][xy]}^{[xy]}.$$ 
Clearly we can choose $r_{[xy]}^{[xy]}=0$. Multiplying the congruence (\ref{eq:[x^2y][xy]}) by $[xy^2]$ we get $[xy]([x^2y][xy^2])\equiv 0$, hence we can choose  
$$r_{[x^2y][xy^2]}^{[xy]}:=t(xy^2)\cdot r_{[x^2y]}^{[xy]}.$$
Similarly, it is easy to see that for the remaining $s\in\trysecondary$ and $q\in\mingen\setminus\primary$ 
the element $r_s^q$ can be chosen from the ideal generated by the $5$ elements of 
$K_{3,2}$ introduced already. 
It follows by Lemma~\ref{lemma:congruences} that $\trysecondary$ is a system of secondary generators 
of $R_{3,2}$, and by Lemma~\ref{lemma:rsq} the ideal $K_{3,2}$ is generated by 
$$r_{[x^2y]}^{[xy]}, \quad r_{[xy^2]}^{[xy]},\quad r_{[x^2y]}^{[x^2y]}, \quad r_{[xy][xy]}^{[xy]}, \quad r_{[xy^2]}^{[xy^2]}.$$ 
Moreover, since $d_{(3,2)}(2)+d_{(4,2)}(2)=2+3=5$, the above is a minimal system of generators 
of the ideal $K_{3,2}$.  


\subsection{The case $m=3$}\label{subsec:m=3} 

In Table~\ref{table:monomialsm=3} we collect the monomials in the elements of $\mingen\setminus\primary$  of descending multidegree $\alpha$ with all $\alpha_i>0$, up to total degree $8$. Monomials congruent to $0$ are indicated by $\star$ (and we indicate by $\star$ the multidegrees where all monomials are congruent to $0$), and the symbol $\sim$ indicates that 
some non-zero scalar multiples of the given monomials are congruent modulo $(\primary)$. 
Table~\ref{table:R_m=3}  should be interpreted as follows: its second line for example says that in multidegree 
$(3,1,1)$ we have the congruence  $[x^2y][xz]+[x^2z][xy]\equiv 0$. Hence we may choose a multihomogeneous  element 
$r_{3,1,1}\in K_{3,3}$ of multidegree $(3,1,1)$ differing 
from $t(x^2y)t(xz)+t(x^2z)t(xy)$ by an element of the ideal of $\freeinvar_{3,3}^3$ generated by $t(x^i),t(y^i),t(z^i)$, $i=1,2,3$. (From now on we 
 change the notation for the relations, the lower indices indicate their multidegree.)  

Note that $\Span_{\mc}(\primary)$  and the ideal $(\primary)$ are not $\gl_m$-submodules in $R_{n,m}$. However, they are $S_m$-submodules, where we 
think of the symmetric group $S_m$  as the subgroup  of  $\gl_m$  consisting of permutation matrices. 
Observe that some of the congruence in Table~\ref{table:R_m=3}  are symmetric or skew symmetric in two variables, and some is symmetric in $x,y,z$. 
We may assume that the corresponding elements $r_{3,1,1}$, $r_{2,2,1}^{(1)}$, etc. are chosen so that they also have the corresponding symmetry or skew-symmetry. 
The last two columns of Table~\ref{table:R_m=3} contain  the number of $S_3$-translates (resp. 
$S_4$-translates) of the relation listed in the third column, where we count the $S_m$-translates up to non-zero scalar multiples, so by the above observation the number of 
$S_m$-translates of $r$ equals the index in $S_m$ of the stabilizer of the congruence corresponding to $r$. 

Denote by $\basis$ the relations listed in the third column of Table~\ref{table:R_m=3} and all their $S_3$-translates. Note that the sum of the numbers in the last but one column of Table~\ref{table:R_m=3} equals the cardinality of $\basis$, so $|\basis|=43$.

Using the $S_3$-translates of the relations in Table~\ref{table:R_m=3} one can easily justify the 
$^\star$ symbols and the equivalences $\sim$ in Table~\ref{table:monomialsm=3}. 
This means that 
up to total degree $8$, all monomials (having descending multidegree) in $\mingen\setminus\primary$ can be reduced to linear combinations of the monomials given in Table~\ref{table:S_m=3}.  (For multidegrees with $\alpha_3=0$, this was shown already in section ~\ref{subsec:m=2}.) 
One can easily see from (\ref{eq:hilbS}) that for each descending multidegree $\alpha$, the number of 
elements in Table~\ref{table:S_m=3} with multidegree $\alpha$ coincides with the coefficient of 
$t_1^{\alpha_1}t_2^{\alpha_2}t_3^{\alpha_3}$ in $H(\secondary;t_1,t_2,t_3)$. 
Define $\trysecondary$ as follows: in descending multidegrees its  
elements are listed in Table~\ref{table:S_m=3}, and 
if $\beta$ is a multidegree in the $S_3$-orbit of some descending multidegree $\alpha$, then choose 
a permutation $\pi\in S_3$ with $\beta_i=\alpha_{\pi(i)}$, and include in $\trysecondary$ the images under $\pi$ of the elements of multidegree $\alpha$ in Table~\ref{table:S_m=3}. (Of course, the set $\trysecondary$ is not uniquely defined: for certain multidegrees, say for multidegree $(1,3,1)$ we may choose for $\pi$ the transposition $(12)$ or the three-cycle $(123)$. However, this does not influence the arguments below.)  
Since the set $\basis$ is (essentially) $S_3$-stable, it follows that up to degree $\leq 8$, all monomials in $\mingen\setminus\primary$ can be reduced to linear combinations of $\trysecondary$ using the relations in $\basis$. Moreover, $H(\trysecondary;t_1,t_2,t_3)=H(\secondary; t_1,t_2,t_3)$. 
Consequently, by  Lemmas~\ref{lemma:congruences} and \ref{lemma:rsq}, $\trysecondary$ is a system of secondary generators, and $\basis$ generates the ideal $K_{3,3}$ up to degree $8$. We know from Proposition~\ref{prop:general} that $K_{3,3}$ is generated in degree $\leq 8$, hence $\basis$ generates $K_{3,3}$.  
Since $d_{(3,2)}(3)+d_{(4,2)}(3)+d_{(2,2,2)}(3)=15+27+1=43=|\basis|$, 
we conclude that $\basis$ is a minimal system of generators of the ideal $K_{3,3}$. 

We finish this Section with the proof of  the congruences in Table~\ref{table:R_m=3}. The relations $r_{3,2}$, $r_{4,2}$, 
$r_{3,3}$ were explained in section~\ref{subsec:m=2}. 
The relation $\varphi(\Psi(w_1,w_2,w_3,w_4))=0$ implies a congruence of multidegree 
${\mathrm{multideg}}(w_1w_2w_3w_4)$  of the form 
$$[w_1\cdot w_2\cdot w_3\cdot w_4]\equiv \cdots$$ where $\cdots$ is a linear combination of monomials in elements $[u]$ with $\deg(u)<\deg(w_1w_2w_3w_4)$, 
which can be written as a polynomial in the elements of $\mingen\setminus\primary$ using (\ref{eq:x^2yz}) below.

\paragraph{$r_{3,1,1}$:} $\varphi(\Psi(x^2,x,y,z))=0$ and $[x^3yz]\equiv 0$ (by Lemma~\ref{lemma:x^nequiv0}) imply 
$$0\equiv 6[x^3yz]=6[x^2\cdot x\cdot y\cdot z]\equiv [x^2y][xz]+[x^2z][xy].$$  

\paragraph{$r_{2,2,1}^{(1)}$, $r_{2,2,1}^{(2)}$:}  
Eliminate $[x^2y^2z]$ from the following consequences of the fundamental relation: 
\begin{equation*}\label{eq:[221]1}
6[ x\cdot y\cdot z\cdot xy]\equiv 3[xy][xyz]+[xz][xy^2]+[x^2y][yz]
\end{equation*}
\begin{equation*}\label{eq:[221]2}
6[ x^2\cdot y\cdot y\cdot z]\equiv 2[x^2y][yz]
\end{equation*}
\begin{equation*}\label{eq:[221]3}
6[ x\cdot x\cdot y^2\cdot z]\equiv 2[xy^2][xz]
\end{equation*}

\paragraph{$r_{4,1,1}$: } 
We have 
$0\equiv [x^4yz]=[x^2\cdot x^2\cdot y\cdot z]\equiv\frac 13[x^2y][x^2z]$. 

\paragraph{$r_{3,2,1}^{(2)}$:} Follows by 
$0\equiv 6[x^3y^2z]=6[x^2\cdot x\cdot y^2\cdot z]\equiv [x^2y^2][xz]+[x^2z][xy^2]$
and (\ref{eq:x^2y^2}). 

\paragraph{$r_{3,2,1}^{(1)}$:} The congruences 
$0\equiv 6[x^3y^2z]=6[xyz\cdot x\cdot x\cdot y]\equiv 2[x^2y][xyz]+2[x^2yz][xy],$
and 
\begin{equation}\label{eq:x^2yz} [x^2yz]\equiv \frac 13 [xy][xz]
\end{equation} 
yield 
$[xy]^2[xz]+3[xyz][x^2y]\equiv 0$, and this and $r_{3,2,1}^{(2)}$ 
imply $r_{3,2,1}^{(1)}$. 

\paragraph{$r_{2,2,2}^{(1)}$, $r_{2,2,2}^{(2)}$:} Eliminate $[x^2y^2z^2]$ from the congruences 
\begin{align*}6[x\cdot x\cdot y^2\cdot z^2]&\equiv 2[xy^2][xz^2]\\
6[x^2\cdot y\cdot y\cdot z^2]&\equiv 2[x^2y][yz^2]\\
6[xyz\cdot x\cdot y\cdot z] &\equiv 2[xyz]^2+[x^2yz][yz]+[xy^2z][xz]+[xyz^2][xy]\\
6[xy\cdot x\cdot y\cdot z^2]&\equiv 3[xy][xyz^2]+[x^2y][yz^2]+[xy^2][xz^2]
\end{align*}
and use (\ref{eq:x^2yz}). 


\subsection{The case $m=4$} \label{subsec:m=4} 

The arguments are the same as in section~\ref{subsec:m=3}. 
We just give the corresponding tables and prove the new congruences (involving all the four variables). 
Tables~\ref{table:R_m=4}, \ref{table:S_m=4} and \ref{table:monomialsm=4} deal only with multidegrees $\alpha$ with all $\alpha_1,\alpha_2,\alpha_3,\alpha_4>0$, since the remaining multidegrees have been taken care in section~\ref{subsec:m=3}. 

The congruence in Table~\ref{table:R_m=4} corresponding to $r_{2,1,1,1}$ is symmetric in the variables $z,w$, hence the number of its $S_4$-translates is 
$\frac{24}{2}=12$. The same holds for $r_{3,1,1,1}$ and $r_{2,2,1,1}^{(3)}$. 
The congruence corresponding to $r_{2,2,1,1}^{(1)}$ is symmetric in the variables $x,y$ and also in the variables $z,w$, hence the number of its $S_4$-translates is 
$\frac{24}{2\cdot 2}=6$. The congruence corresponding to $r_{2,2,1,1}^{(2)}$ is unchanged if we simultaneously interchange $x,y$ and $z,w$, hence the number of its $S_4$-translates is 
$\frac{24}{2}=12$. 

The set $\basis$ of all $S_4$-translates of the relations listed in Tables~\ref{table:R_m=3} and \ref{table:R_m=4} has cardinality $|\basis|=196$. 
(This agrees with $d_{(3,2)}(4)+d_{(4,2)}(4)+d_{(2,2,2)}=60+126+10=196$.) 
Table~\ref{table:S_m=4} together with Table~\ref{table:S_m=3} give a system of secondary generators up to degree $8$ in descending multidegrees.  

An inspection of Tables~\ref{table:R_m=4} and \ref{table:monomialsm=4} shows that up to degree $8$, all monomials in the elements of $\mingen\setminus\primary$ with descending multidegree can be reduced to a linear combination of elements listed in Table~\ref{table:S_m=4} using the $S_4$-translates of the relations in Tables~\ref{table:R_m=3} and \ref{table:R_m=4}. 
We just give some sample examples: 
$$[x^2y][xz][xw]\equiv -[x^2z][xy][xw]\equiv [x^2w][xy][xz]\equiv -[x^2y][xz][xw]$$ 
by the relations $r_{3,1,1}$, $r_{3,0,1,1}$, $r_{3,1,0,1}$, 
hence all the above products are congruent to $0$. 
The relations $\sim$ in multidegree $(3,3,1,1)$ can be derived as follows 
(at each congruence we indicate the relation whose $S_m$-translate is used): 
\begin{align*}
&-\frac 13[xy]^2[xz][yw]
\stackrel{r_{3,2,1}^{(1)},r_{3,2,1}^{(2)}}\equiv [x^2y][xyz][yw]
\stackrel{r_{3,2,1}^{(1)}}\equiv [x^2z][xy^2][yw]
\stackrel{r_{3,1,1}}\equiv -[x^2z][y^2w][xy] \\
&\stackrel{r_{3,1,1}}\equiv [x^2y][y^2w][xz]
\stackrel{r_{3,2,1}^{(1)}}\equiv [xy^2][xyw][xz]
\stackrel{r_{2,2,1}^{(1)}}\equiv [x^2y][xyw][yz]
\stackrel{r_{3,2,1}^{(1)}}\equiv [x^2w][xy^2][yz] \\
&\stackrel{r_{3,1,1}}\equiv -[x^2w][y^2z][xy] 
\stackrel{r_{3,1,1}}\equiv [x^2y][y^2z][xw]
\stackrel{r_{3,2,1}^{(1)}}\equiv  [xy^2][xyz][xw]
\stackrel{r_{3,2,1}^{(1)},r_{3,2,1}^{(2)}}\equiv -\frac 13 [xy]^2[xw][yz]
\end{align*}
Furthermore, 
$$-[xy]^3[zw]\stackrel{r_{3,3}}\equiv 3[x^2y][xy^2][zw]
\stackrel{r_{2,1,1,1}}\equiv 3[xyz][xy^2][xw]+3[xyw][xy^2][xz]$$
hence by the above long chain of congruences we conclude 
$$[xy]^3[zw]\equiv -6[x^2y][xyz][yw].$$

Finally we verify the four-variable relations. 
\paragraph{$r_{2,1,1,1}$:}  
Various substitutions into the fundamental relation yield 
\begin{equation}
6[x^2\cdot y\cdot z\cdot w]\equiv  [x^2y][zw]+[x^2z][yw]+[x^2w][yz]
\label{eqno1}
\end{equation}
\begin{equation}
 6[x\cdot xy\cdot z\cdot w]\equiv  [x^2y][zw]+2[xy][xzw] +[xz][xyw]+[xw][xyz] 
\label{eqno2}
\end{equation}
\begin{equation}
6[x\cdot y\cdot xz\cdot w]\equiv  [x^2z][yw]+[xy][xzw] +2[xz][xyw]+[xw][xyz] 
\label{eqno3}
\end{equation}
\begin{equation}
 6[x\cdot y\cdot z\cdot xw]\equiv  [x^2w][yz]+[xy][xzw] +[xz][xyw]+2[xw][xyz] 
\label{eqno4}
\end{equation}
Now $\frac{1}{2}(\eqref{eqno1}+\eqref{eqno2}-\eqref{eqno3}-\eqref{eqno4})$  gives 
$ r_{2,1,1,1}$. 
\label{eqno8}

\paragraph{$r_{2,2,1,1}^{(1)}$, $r_{2,2,1,1}^{(2)}$, $r_{2,2,1,1}^{(3)}$:} 
To make calculations more transparent, introduce  the following temporary notation for the monomials of $Q\setminus P$ of multidegree $(2,2,1,1):$ $a_1=[xy]^2[zw],$ $a_2=[xy][xz][yw],$ $a_3=[xy][yz][xw],$ $b_1=[x^2y][yzw],$ $b_2=[xy^2][xzw],$ $c_1=[x^2z][y^2w],$ $c_2=[x^2w][y^2z],$ $d=[xyz][xyw].$ 
Using  (\ref{eq:x^2y^2}), (\ref{eq:x^2yz}), their $S_4$-translates, and 
\begin{equation}\label{eq:xyzw} 
6[xyzw]\equiv [xy][zw]+[xz][yw]+[xw][yz] 
\end{equation}
various substitutions into the fundamental relation yield 

\begin{eqnarray}  
6[xy\cdot xz\cdot y\cdot w] & \equiv  & \frac{1}{3}a_1+\frac{1}{3}a_2+\frac{1}{3}a_3+b_2+d \label{eq:1}\\
6[xy\cdot x\cdot yz\cdot w] & \equiv  &  \frac{1}{3}a_1+\frac{1}{3}a_2+ \frac{1}{3}a_3+ b_1+d \label{eq:2}\\
6[x^2y\cdot y\cdot z\cdot w] & \equiv & \frac 13 a_1+\frac 13 a_2+\frac 13 a_3+2b_1 \label{eq:3}\\
6[xy\cdot xy\cdot z\cdot w] & \equiv  & \frac{2}{3} a_2+\frac{2}{3}a_3+2d \label{eq:4}\\
6[xz\cdot x\cdot yw\cdot y] & \equiv  &  \frac{1}{6} a_1+\frac{1}{2}a_2+\frac{1}{6}a_3+c_1+d \label{eq:5}\\
6[xw\cdot x\cdot yz\cdot y] & \equiv  &  \frac{1}{6}a_1+\frac{1}{6}a_2+\frac{1}{2}a_3+c_2+d \label{eq:6}
\end{eqnarray}

Equations (\ref{eq:1}), (\ref{eq:2}), (\ref{eq:3}) imply 
\begin{equation} b_1\equiv b_2\equiv d\label{eq:b1=b2=d}
\end{equation} 
(in particular, relation $r_{2,2,1,1}^{(3)}$). From  equations 
(\ref{eq:b1=b2=d}), (\ref{eq:1}), and (\ref{eq:4}) we conclude 
\begin{equation}\label{eq:a1=a2+a3}a_1\equiv a_2+a_3. \end{equation} 
Taking the difference of (\ref{eq:4}) and (\ref{eq:5}), and eliminating $a_1$ by (\ref{eq:a1=a2+a3}) 
we get 
\begin{equation}\label{eq:a3} \frac 13 a_3\equiv c_1-d.\end{equation}  
Taking the difference of   (\ref{eq:4}) and (\ref{eq:6}), and eliminating $a_1$ by (\ref{eq:a1=a2+a3}) 
we get 
\begin{equation}\label{eq:a2} \frac 13 a_2\equiv c_2-d, \mbox{ hence }r_{2,2,1,1}^{(2)}.\end{equation}  
Finally, (\ref{eq:a1=a2+a3}), (\ref{eq:a3}), (\ref{eq:a2}) yield 
$$a_1\equiv 3c_1+3c_2-6d, \mbox{ hence }r_{2,2,1,1}^{(1)}.$$  

\paragraph{$r_{3,1,1,1}$:} 
\begin{equation}\label{eq:[3111]1}
0\equiv 6[x^3yzw]=6[x\cdot xy\cdot xz\cdot w]
\equiv \frac{2}{3}[xy][xz][xw]+ [x^2y][xzw]+[x^2z][xyw]
\end{equation} 
Permuting $y,z,w$ we get 
\begin{equation}\label{eq:[3111]2}
0\equiv \frac{2}{3}[xy][xz][xw]+ [x^2y][xzw]+[x^2w][xyz]
\end{equation}
\begin{equation}\label{eq:[3111]3}
0\equiv \frac{2}{3}[xy][xz][xw]+ [x^2z][xyw]+[x^2w][xyz]
\end{equation}
Now $\eqref{eq:[3111]1}+\eqref{eq:[3111]2}-\eqref{eq:[3111]3}$  gives $r_{3,1,1,1}$. 


\section{The proof of Theorem~\ref{thm:lowerbound}}\label{sec:lowerbound} 

First we point out that the kernel $\ker(\varphi_{n,m}^2)$ of the restriction of $\varphi$ to $\freeinvar_{n,m}^2$ is described by the {\it second fundamental theorem for vector invariants of the full orthogonal group} (cf. Theorem 2.17.A in \cite{weyl}).  

\begin{proposition}\label{prop:orthogonal} 
The ideal $\ker(\varphi_{n,m}^2)$ is (minimally) generated by the $\gl_m$-submodule of $\freeinvar_{n,m}^2$ spanned by $J$. 
\end{proposition} 

\begin{proof} Denote by $\ooo(V)$ the orthogonal group, i.e. $\ooo(V)$ consists of the linear transformations of $V=\mc^n$ preserving the standard quadratic form 
$(v_1,\ldots,v_n)\mapsto \sum_{i=1}^n v_i^2$. The orthogonal complement $V_0$ of $(1,\ldots,1)\in V$ consists of the vectors in $V$ with zero coordinate sum. We identify the stabilizer of $(1,\ldots,1)$ in $\ooo(V)$ with $\ooo(V_0)$ in the obvious way. Note that the elements of $S_n$ as transformations on $V$ do belong to $\ooo(V_0)$. As an immediate corollary of the {\it first fundamental theorem on vector invariants of the orthogonal group} (cf. Theorem 2.11.A in \cite{weyl}) 
we conclude that 
$\varphi(\freeinvar_{n,m}^2)=\mc[V^m]^{\ooo(V_0)}\subset R_{n,m}$. 
Set  $b_{ij}:=[x_ix_j]-\frac 1n[x_i][x_j]$ for $1\leq i,j\leq m$. 
The projection from $V\to V_0$ with kernel spanned by $(1,\ldots,1)$ induces an 
identification of $\mc[V_0^m]^{\ooo(V_0)}$ with the  subalgebra of $\mc[V^m]^{\ooo(V_0)}$ generated by the $b_{ij}$. Moreover,  $\mc[V^m]^{\ooo(V_0)}$ is an $m$-variable polynomial ring over 
$\mc[V_0^m]^{\ooo(V_0)}$ generated by $[x_1],\ldots,[x_m]$. 
Denote by $L$ the subalgebra of $\freeinvar_{n,m}^2$ generated by 
$u_{ij}:=t(x_ix_j)-\frac 1n t(x_i)t(x_j)$, $1\leq i,j\leq n$. By the above considerations, the ideal 
$\ker(\varphi_{n,m}^2)$ is generated by the kernel of the restriction $\varphi\vert_L$ of $\varphi$ to $L$. 
Now the kernel of $\varphi\vert_L:L\to \mc[V_0^m]^{\ooo(V_0)}$, $u_{ij}\mapsto b_{ij}$ is given by 
Theorem 2.17.A in \cite{weyl}, stating that it is generated by the polarizations of $J$. 
\end{proof}

Let $I$ be a highest weight vector in $\ker(\varphi)$ of weight $2^n=(2,\ldots,2)$.  
We claim that $I$ necessarily belongs to $\freeinvar_{n,m}^2$, and if $I$ is contained in 
$\freeinvar^{(+)}\cdot \ker(\varphi)$, then $I$ is necessarily contained in 
$(\freeinvar_{n,m}^2)^{(+)}\cdot \ker(\varphi_{n,m}^2)$. It is clear that Theorem~\ref{thm:lowerbound} follows from this claim and Proposition~\ref{prop:orthogonal}. 

To prove this claim, given a polynomial $\gl_m$-module $U$ and a partition $\lambda\in\partitions_m$, denote by $\lambda(U)$ the $\lambda$-isotypic component of $U$ (i.e. the sum of the $\gl_m$-submodules of $U$ isomorphic to the irreducible $\gl_m$-module $V_{\lambda}$). 
Write $\lambda\subset\mu$ (where $\lambda,\mu\in\partitions_m$) if $\lambda_i\leq\mu_i$ 
for $i=1,\ldots,m$.   
It follows from Pieri's rule that denoting by $A$ the ideal in $\freeinvar$ generated by the $t(w)$ with $\deg(w)\geq 3$, we have $A\subset \sum_{(3)\subset \lambda}
\lambda(\freeinvar)$. Since $\freeinvar_{n,m}^2$ is a $\gl_m$-module direct complement of $A$, the 
$2^n$-isotypic component of $\freeinvar$ is contained in $\freeinvar_{n,m}^2$. In particular, 
$I$ belongs to $\freeinvar_{n,m}^2$.  

Again by Pieri's rule, 
the ideal generated by $\lambda(\freeinvar)$ is contained in $\sum_{\lambda\subset\mu}\mu(\freeinvar)$. 
So if $I\in\freeinvar^{(+)}\cdot\ker(\varphi)$, then 
$I\in\sum_{\lambda\subsetneq 2^n} \freeinvar^{(+)}\lambda(\ker(\varphi))$. 
Again since $\freeinvar_{n,m}^2$ is a $\gl_m$-module complement of $A$, we conclude that 
$\lambda(\ker(\varphi))\leq\freeinvar_{n,m}^2$ whenever $\lambda\subsetneq 2^n$, 
hence $I\in\sum_{\lambda\subsetneq 2^n} \freeinvar^{(+)}\lambda(\ker(\varphi_{n,m}^2))$. 
Using the retraction $\freeinvar\to\freeinvar_{n,m}^2$ with kernel $A$, we conclude that $I$ is contained in $(\freeinvar_{n,m}^2)^{(+)}\cdot\ker(\varphi_{n,m}^2)$. 

\begin{table}[p]
\begin{center}
\begin{tabular}{c|c}
Multidegree & $3$-variable monomials  in the elements of $\mingen\setminus\primary$ \\ \hline \hline 
$(1,1,1)$ & $[xyz]$ \\ \hline \hline 
$(2,1,1)$ & $[xy][xz]$ \\ \hline \hline
$(3,1,1)$ & $[x^2y][xz]\sim [x^2z][xy]$ \\ \hline 
$(2,2,1)$ &  $[x^2y][yz]\sim [xy^2][xz]$,  $[xyz][xy]^\star$ \\ \hline \hline
$(4,1,1)^\star$ & $[x^2y][x^2z]^\star$ \\ \hline 
$(3,2,1)$ &  $[x^2y][xyz]\sim [x^2z][xy^2] \sim [xy]^2[xz]$ \\ \hline 
$(2,2,2)$ &  $[xyz]^2\sim [x^2y][yz^2]\sim [x^2z][y^2z] \sim [xy^2][xz^2] $,  $[xy][xz][yz]^\star$  \\ \hline \hline 
$(4,2,1)^\star$ &  $[xy]^2[x^2z]^\star$,  $[x^2y][xy][xz]^\star$ \\ \hline 
 $(3,3,1)^\star$ & $[xyz][xy]^2{}^\star$, $[x^2y][xy][yz]^\star $, $[xy^2][xy][xz]^\star$ \\ \hline
 $(3,2,2)$ & $[x^2y][xz][yz]\sim [xy^2][xz]^2\sim [x^2z][xy][yz] \sim [xy]^2[xz^2]$, $[xyz][xy][xz]^\star$, \\ \hline \hline 
 $(5,2,1)^\star$ & $[x^2y]^2[xz]^{\star}$, $[x^2y][x^2z][xy]^\star$ \\ \hline
 $(4,3,1)^\star$ & $[x^2y][x^2y][yz]^\star$, $[x^2y][xyz][xy]^\star$, $[x^2y][xy^2][xz]^\star$, $[x^2z][xy^2][xy]^\star$, $[xy]^3[xz]^\star$ \\ \hline 
$(4,2,2)^\star$ & 
\begin{tabular}{c} 
$[x^2y][x^2z][yz]^\star$, $[x^2y][xyz][xz]^\star$, $[x^2y][xy][xz^2]^\star$, \\ 
$[x^2z][xy^2][xz]^\star$, $[x^2z][xyz][xy]^\star$, 
$[xy]^2[xz]^2{}^\star$ 
\end{tabular}
\\ \hline 
$(3,3,2)^\star$ & 
\begin{tabular}{c} 
$[x^2y][xyz][yz]^\star$, $[x^2y][xy][yz^2]^\star$, $[x^2y][xz][y^2z]^\star$, $[x^2z][xy^2][yz]^\star$, \\
$[x^2z][xy][y^2z]^\star$, $[xy^2][xyz][xz]^\star$, $[xy^2][xy][xz^2]^\star$, $[xyz]^2[xy]^\star$, $[xy]^2[xz][yz]^\star$ 
\end{tabular}
  \end{tabular}
\end{center}
\caption[]{$3$-variable monomials in $\mingen\setminus\primary$}
\label{table:monomialsm=3}
\end{table}

\begin{table}[hbtp]
\begin{center}
\begin{tabular}{c|c|c|c|c}
Multidegree & Congruence & Relation & $\sharp\{S_3\mbox{ translates}\}$ &  $\sharp\{S_4\mbox{ translates}\}$\\ \hline \hline 
$(3,2,0)$ & $[xy][x^2y]\equiv 0$ & $r_{3,2}$ & $6$  & $12$ \\ \hline 
$(3,1,1)$ & $[x^2y][xz]+[x^2z][xy]\equiv 0$ & $r_{3,1,1}$ & $3$ & $12$ \\ \hline 
$(2,2,1)$ & $[x^2y][yz]-[xy^2][xz]\equiv 0$ & $r_{2,2,1}^{(1)}$ & $3$ & $12$\\ \hline 
$(2,2,1)$ & $[xy][xyz]\equiv 0$ & $r_{2,2,1}^{(2)}$ & $3$ &  $12$ \\ \hline \hline
$(4,2,0)$ & $[x^2y]^2\equiv 0$ & $r_{4,2}$ & $6$ & $12$ \\ \hline 
$(4,1,1)$ & $[x^2y][x^2z]\equiv 0$ & $r_{4,1,1}$ & $3$ & $12$ \\ \hline 
$(3,3,0)$ & $[xy]^3+3[x^2y][xy^2]\equiv 0$ & $r_{3,3}$ & $3$ & $6$ \\ \hline 
$(3,2,1)$ & $[x^2y][xyz]-[x^2z][xy^2]\equiv 0$ & $r_{3,2,1}^{(1)}$ & $6$  & $24$ \\ \hline 
$(3,2,1)$ & $[xy]^2[xz]+3[x^2z][xy^2]\equiv 0$ & $r_{3,2,1}^{(2)}$ & $6$ & $24$  \\ \hline 
$(2,2,2)$ & $[xy][yz][zx]\equiv 0$ & $r_{2,2,2}^{(1)}$ & $1$ & $4$ \\ \hline 
$(2,2,2)$ & $[xyz]^2-[xy^2][xz^2]\equiv 0$ & $r_{2,2,2}^{(2)}$ & $3$ & $12$ 
\end{tabular}
\end{center}
\caption[Relations in the case $m=3$]{Relations in the case $m=3$}
\label{table:R_m=3}
\end{table}

\begin{table}[hbtp]
\begin{center}
\begin{tabular}{c|c|l}
Degree & Multidegree & Secondary generators \\ \hline \hline 
$0$ & $(0,0,0)$ & $1$ \\ \hline 
$2$ & $(1,1,0)$ & $[xy]$ \\ \hline 
$3$ & $(2,1,0)$ & $[x^2y]$  \\ \hline 
$3$ & $(1,1,1)$ & $[xyz]$ \\ \hline 
$4$ & $(2,2,0)$ & $[xy]^2$  \\ \hline 
$4$ & $(2,1,1)$ & $[xy][xz]$ \\ \hline 
$5$ & $(3,1,1)$ & $[x^2y][xz]$ \\ \hline 
$5$ & $(2,2,1)$ &  $[x^2y][yz]$ \\ \hline 
$6$ & $(3,3,0)$ & $[x^2y][xy^2]$ \\ \hline 
$6$ & $(3,2,1)$ & $[x^2y][xyz]$   \\ \hline 
$6$ & $(2,2,2)$ & $[xyz]^2$ \\ \hline 
$7$ & $(3,2,2)$ & $[x^2y][xz][yz]$
\end{tabular}
\end{center}
\caption[Secondary generators in the case $m=3$]{Secondary generators in the case $m=3$}
\label{table:S_m=3}
\end{table}

\begin{table}[hbtp]
\begin{center}
\begin{tabular}{c|c|c|c}
Multidegree & Congruence & Relation & $\sharp$ of $S_4$-translates \\ \hline \hline 
$(2,1,1,1)$ & $[x^2y][zw]\equiv  [xyz][xw]+[xyw][xz]$ & $r_{2,1,1,1}$ & $12$ \\ \hline 
$(3,1,1,1)$ & $[xy][xz][xw]\equiv -3[x^2y][xzw]$ & $r_{3,1,1,1}$ & $12$ \\ \hline 
$(2,2,1,1)$ & $[xy]^2[zw]\equiv 3[x^2z][y^2w]+3[x^2w][y^2z]-6[xyz][xyw]$ & $r_{2,2,1,1}^{(1)}$ & $6$ \\ \hline 
$(2,2,1,1)$ & $[xy][xz][yw]\equiv 3[x^2w][y^2z]-3[xyz][xyw]$ & $r_{2,2,1,1}^{(2)}$ & $12$ \\ \hline 
$(2,2,1,1)$ & $[x^2y][yzw]\equiv [xyz][xyw]$ & $r_{2,2,1,1}^{(3)}$ & $12$
\end{tabular}
\end{center}
\caption[Relations in the case $m=4$]{$4$-variable relations in the case $m=4$}
\label{table:R_m=4}
\end{table}

\begin{table}[hbtp]
\begin{center}
\begin{tabular}{c|l}
 Multidegree & Secondary generators \\ \hline \hline 
 $(1,1,1,1)$ & $[xy][zw]$, $[xz][yw]$, $[xw][yz]$ \\ \hline 
$(2,1,1,1)$ & $[xy][xzw]$, $[xyw][xz]$, $[xyz][xw]$ \\ \hline  
$(3,1,1,1)$ & $[x^2y][xzw]$  \\ \hline 
$(2,2,1,1)$ & $[x^2z][y^2w]$, $[x^2w][y^2z]$, $[xyz][xyw]$ \\ \hline 
$(3,2,1,1)$ & $[x^2y][xz][yw]$  \\ \hline 
$(2,2,2,1)$ & $[xy][xzw][yz]$, $[xyw][xz][yz]$, $[xy][xz][yzw]$ \\ \hline 
 $(3,3,1,1)$ & $[x^2y][xyz][yw]$ \\ \hline 
$(3,2,2,1)$ &  $[xyz]^2[xw]$ \\ \hline 
$(2,2,2,2)$ & $[x^2y][yz][zw^2]$, $[x^2z][y^2w][zw]$, $[x^2w][y^2z][zw]$
\end{tabular}
\end{center}
\caption[Secondary generators in the case $m=4$]{$4$-variable secondary generators}
\label{table:S_m=4}
\end{table}

\begin{table}[hbtp]
\begin{center}
\begin{tabular}{c|c}
Multidegree & $4$-variable monomials  in the elements of $\mingen\setminus\primary$\\ \hline \hline 
$(1,1,1,1)$ & $[xy][zw]$, $[xz][yw]$, $[xw][yz]$\\
\hline\hline 
$(2,1,1,1)$ & \begin{tabular}{c} $[xy][xzw]$, $[xz][xyw]$, $[xw][xyz]$, \\ 
$[x^2y][zw]$, $[x^2z][yw]$, $[x^2w][yz]$ \end{tabular} 
\\ \hline \hline
$(3,1,1,1)$ & $[x^2y][xzw]\sim [x^2z][xyw]\sim [x^2w][xyz]\sim [xy][xz][xw]$ \\ \hline 
$(2,2,1,1)$ & \begin{tabular}{c} $[xyz][xyw]\sim [x^2y][yzw] \sim [xy^2][xzw]$, $[x^2z][y^2w]$, $[x^2w][y^2z]$, \\ 
 $[xy]^2[zw]$, $[xy][xz][yw]$, $[xy][xw][yz]$ 
 \end{tabular} \\ \hline \hline
$(4,1,1,1)^\star$ &  $[x^2y][xz][xw]^\star$, $[x^2z][xy][xw]^\star$,  $[x^2w][xy][xz]^\star$ \\ \hline 
$(3,2,1,1)$ & 
\begin{tabular}{c} 
 $[x^2y][xz][yw] \sim [x^2y][xw][yz]\sim [x^2z][xy][yw]\sim$ \\ $\sim [x^2w][xy][yz]\sim [xy^2][xz][xw] 
 \sim [xy]^2[xzw]$, \\ 
$[x^2y][xy][zw]^\star$, $[xyz][xy][xw]^\star$, $[xyw][xy][xz]^\star$ 
\end{tabular} \\ \hline 
$(2,2,2,1)$ & 
\begin{tabular}{c} $[xy][yz][xzw]\sim [x^2z][yw][yz]\sim [xz^2][xy][yw]$, \\ 
$[xz][yz][xyw]\sim [x^2y][yz][zw]\sim [xy^2][xz][zw]$, \\ 
$[xz][xy][yzw]\sim [xz][xw][y^2z] \sim [xy][xw][yz^2]$, \\
$[x^2w][yz]^2$, $[xy]^2[z^2w]$, $[xz]^2[y^2w]$, \\
$[xyz][xy][zw]^\star$,$[xyz][xz][yw]^\star$, $[xyz][xw][yz]^\star$
\end{tabular} 
 \\ \hline \hline
$(5,1,1,1)^\star$ &  $[x^2y][x^2z][xw]^\star$, $[x^2y][x^2w][xz]^\star$, $[x^2z][x^2w][xy]^\star$  \\ \hline 
$(4,2,1,1)^\star$ &  
\begin{tabular}{c} 
$[x^2y]^2[zw]^\star$, $[x^2y][x^2z][yw]^\star$, $[x^2y][x^2w][yz]^\star$, $[x^2y][xyz][xw]^\star$, \\ 
$[x^2y][xyw][xz]^\star$, 
$[x^2y][xy][xzw]^\star$, 
$[x^2z][xy^2][xw]^\star$, $[x^2z][xyw][xy]^\star$, \\ 
$[x^2w][xy^2][xz]^\star$, $[x^2w][xyz][xy]^\star$, 
$[xy]^2[xz][xw]^\star$
\end{tabular}
 \\ \hline 
 $(3,3,1,1)$ & \begin{tabular}{c} 
  $[x^2y][xyz][yw]\sim [x^2y][y^2z][xw]\sim [y^2z][xy][x^2w]\sim 
 [xy^2][yz][x^2w]\sim$ \\ $\sim [x^2y][yz][xyw] 
 \sim [xy^2][xz][xyw]\sim [x^2y][xz][y^2w]\sim [x^2z][xy][y^2w] \sim$ \\  
 $\sim [x^2z][xy^2][yw]\sim [xy^2][xyz][xw]  \sim  [xy]^2[xz][yw]\sim 
[xy]^2[yz][xw]\sim $\\ 
$\sim [x^2y][xy^2][zw]\sim [xy]^3[zw]$,\\ 
$[xyz][xy][xyw]^\star$, 
 $[x^2y][xy][yzw]^\star$, 
 $[xy^2][xy][xzw]^\star$ 
 \end{tabular} 
  \\ \hline
 $(3,2,2,1)$ & \begin{tabular}{c} $[xyz]^2[xw]\sim 
 [x^2y][xyz][zw]\sim  [x^2z][xyz][yw]\sim [x^2z][xy^2][zw]\sim$ \\
$\sim [x^2y][xz^2][yw]\sim  [x^2y][yz^2][xw]\sim  [x^2z][y^2z][xw] \sim [xy^2][xz^2][xw]\sim $
\\ $\sim [y^2z][xz][x^2w]\sim  [yz^2][xy][x^2w] \sim [xy]^2[xz][zw]\sim [xz]^2[xy][yw]$, \\
  $[xy][xz][yz][xw]^\star$,     $[x^2y][xz][yzw]^\star$,  $[x^2z][xy][yzw]^\star$, 
  $[x^2y][yz][xzw]^\star$, \\  $[x^2z][yz][xyw]^\star$ \
   $[x^2y][xy][z^2w]^\star$,  $[x^2z][xz][y^2w]^\star$, 
  $[xy^2][xz][xzw]^\star$, \\  $[xyz][xy][xzw]^\star$, 
  $[xz^2][xy][xyw]^\star$,  $[xyz][xz][xyw]^\star$,
  $[xyz][yz][x^2w]^\star$
  \end{tabular} 
 \\ \hline 
$(2,2,2,2)$ & 
\begin{tabular}{c} 
$[xz][yz][xw][yw]\sim [x^2y][yz][zw^2]\sim [xy^2][xz][zw^2]\sim  [x^2y][yw][z^2w]\sim [xy^2][xw][z^2w]$,\\
 $[xy][yz][xw][zw] \sim[x^2z][y^2w][zw]\sim [x^2z][yz][yw^2]\sim [xz^2][xy][yw^2]\sim  [xz^2][xw][y^2w]$,  \\
$[xy][xz][yw][zw]\sim [x^2w][y^2z][zw]\sim [x^2w][yz^2][yw]\sim [xz][xw^2][y^2z]\sim [yz^2][xy][xw^2]$, \\
 $[xyz][xy][zw^2]^\star$,  $[xyz][xz][yw^2]^\star$, $[xyz][yz][xw^2]^\star$,  $[x^2y][zw][yzw]^\star$, \\
 $[x^2z][yw][yzw]^\star$,  $[xy][xzw][yzw]^\star$, $[xyw][xz][yzw]^\star$, $[xyz][xw][yzw]^\star$ \\
 $[xy^2][zw][xzw]^\star$, $[y^2z][xw][xzw]^\star$, 
$[xyw][yz][xzw]^\star$,  $[xyz][yw][xzw]^\star$, \\ 
$[xz^2][yw][xyw]^\star$, $[yz^2][xw][xyw]^\star$, 
$[xyz][zw][xyw]^\star$, $[xyw][xy][z^2w]^\star$, \\ $[xzw][xz][y^2w]^\star$, 
 $[yzw][yz][x^2w]^\star$, 
 \\ $[xy]^2[zw]^2$, $[xz]^2[yw]^2$,  
 $[yz]^2[xw]^2$
\end{tabular}
  \end{tabular}
\end{center}
\caption[]{$4$-variable monomials in $\mingen\setminus\primary$}
\label{table:monomialsm=4}
\end{table}


\end{document}